\documentclass[aop,noinfoline
]{imsart}
\usepackage{amscd,amsfonts,amssymb,amsmath,amsthm,latexsym,array,hhline,mathrsfs,booktabs,fancybox,calc,textcomp,xcolor,natbib,graphicx,parskip
}
\usepackage[sans]{dsfont}
\usepackage[latin1]{inputenc}
\setattribute{journal}{name}{}

\numberwithin{equation}{section}

\newtheorem{theorem}{Theorem}[section]
\newtheorem{lemma}[theorem]{Lemma}
\newtheorem{proposition}[theorem]{Proposition}
\newtheorem{corollary}[theorem]{Corollary}

\newenvironment{remark}[1][Remark.]{\begin{trivlist}
\item[\hskip \labelsep {\scshape #1}]}{\end{trivlist}}

\def\TV{\operatorname{TV}}

\def\B{\mathbb{B}}
\def\D{\mathbb{D}}
\def\E{\mathbf{E}}
\def\G{\mathbb{G}}

\def\H{\mathbb{H}}
\def\N{\mathbb{N}}
\def\Q{\mathbb{Q}}
\def\R{\mathbb{R}}
\def\S{\mathbb{S}}

\def\GG{\mathcal{G}}
\def\HH{\mathcal{H}}
\def\RR{\mathscr{R}}

\def\NN{\mathcal{N}}

\def\CC{\mathcal{C}}

\def\SS{\mathcal{S}}

\def\H{\mathbb{H}}
\def\N{\mathbb{N}}
\def\P{\mathbf{P}}

\renewcommand{\d}{\mathrm{d}}
\newcommand{\e}{\mathrm{e}}

\newcommand{\Var}{\operatorname{Var}}

\newcommand{\BB}{\mathcal{B}}
\newcommand{\DD}{\mathcal{D}}

\newcommand{\FF}{\mathcal{F}}
\renewcommand{\GG}{\mathcal{G}}

\newcommand{\ep}{\varepsilon}

\renewcommand{\tt}{^{t}}
\renewcommand{\subseteq}{\subset}
\newcommand{\WW}{\mathcal{W}}
\renewcommand{\hat}{\widehat}
\renewcommand{\Upsilon}{Z}
\renewcommand{\tilde}{\widetilde}

\begin{document}

\begin{frontmatter}

\title{Uniform central limit theorems for multidimensional diffusions}
\runtitle{Donsker theorems for multivariate diffusions}
%
\begin{aug}
  \author{\fnms{Angelika}  \snm{Rohde}
\ead[label=e1]{angelika.rohde@math.uni-hamburg.de}} 
\and
  \author{\fnms{Claudia} \snm{Strauch}\ead[label=e2]{claudia.strauch@math.uni-hamburg.de}}
%
%
  \runauthor{A. Rohde and C. Strauch}
  \affiliation{Universit\"at Hamburg}
%
%
%
\end{aug}
\begin{abstract}  
It has recently been shown that there are substantial differences in the regularity behavior of the empirical process based on scalar diffusions as compared to the
classical empirical process, due to the existence of diffusion local time. 
Besides establishing strong parallels to classical theory such as Ossiander's bracketing CLT and the ge\-ne\-ral Gin\'e--Zinn CLT for uniformly bounded 
families of functions, we find increased regularity also for multivariate ergodic diffusions, assuming that the invariant measure is finite with Lebesgue density $\pi$. 
The effect is diminishing for growing dimension but always present. 
The fine differences to the classical iid setting are worked out using exponential inequalities for martingales and additive functionals
of continuous Markov processes as well as the characterization of the sample path behavior of Gaussian processes by means of the generic chaining bound.
To uncover the phenomenon, we study a smoothed version of the empirical diffusion process. 
It turns out that uniform weak convergence of the smoothed empirical diffusion process under necessary and sufficient conditions can take place 
with even exponentially small bandwidth in dimension $d=2$, and with strongly undersmoothing bandwidth choice for parameters $\beta > d/2$ in case $d\geq 3$, 
assuming that the coordinates of drift and diffusion coefficient belong to some H\"older ball with parameter $\beta$. 
\end{abstract}
%
%

\end{frontmatter}

\section{Introduction}

\small

Let $X$ be a stationary, ergodic diffusion process on $E \subseteq \R^d$ with invariant probability measure $\mu$,
and denote its infinitesimal generator on $L^2(E,\d\mu)$ by $A$. We refer to Section \ref{sec_pre} for precise definitions.
The functional central limit theorem for stationary ergodic Markov processes due to \cite{bhat82} states that for any fixed $t\geq0$ and any fixed function $f$ of the form
$f = Ag$,
\begin{equation}\nonumber
	\frac{1}{\sqrt n}\bigg(\int_0^{nt} f(X_u)\d u\bigg)_{t\geq 0}\ \rightarrow_{\DD}\  \sigma(f)W \hspace*{10mm} (n \rightarrow \infty)
\end{equation}
where $\sigma^2(f) := -2\int_E f(x)g(x)\mu(\d x)$ and $W=(W_t)_{t\geq 0}$ denotes a standard Wiener process. In particular, 
$$
\frac{1}{\sqrt{n}}\int_0^n A g(X_u)\d u\ \rightarrow_{\DD}\ Z \sim \NN(0,\sigma^2(f))\hspace*{10mm} (n \rightarrow \infty).
$$ 
The passage to a continuous-time result is obvious. 
Given any finite set of functions $f_1,\ldots,f_m$ solving the Poisson equation $f_i = Ag_i$, $i =1,\ldots,m$, the law of the $m$-dimensional process with components
$t^{-1/2}\int_0^{t}A g_i(X_u)\d u$, $i=1,\ldots,m$, 
converges weakly to an $m$-dimensional centered Gaussian distribution with asymptotic covariances given by
$$-\int_E Ag_i(x)g_j(x)\mu(\d x) - \int_E g_i(x)Ag_j(x)\mu(\d x), \qquad i,j=1,\ldots,m,$$
by means of the Cram\'er-Wold device. 

Classical empirical process theory is concerned with limit results which hold uniformly over entire (pos\-sib\-ly infinite-dimensional) classes of functions. 
The fundamental object of investigation is the classical empirical process in the setting of independent, identically distributed random variables $X_1,...,X_n\sim\P$
$$
\big(\G_n^{classic}(f)\big)_{f\in\FF}\ :=\ \left(\frac{1}{\sqrt{n}}\sum_{i=1}^n\big(f(X_i)-\E f(X_1)\big)\right)_{f\in\FF},
$$
where $\FF$ is a class of functions, typically in $L^2(\d\P)$. 
Lindeberg's CLT gives convergence of the finite-dimensional marginals whenever the variance of $f(X_1)$ is finite.
In order to extend this result to a CLT in $\ell^{\infty}(\FF)$ which holds uniformly over some infinite-dimensional function class,
the existence of the limiting Gaussian process does not suffice.

It was shown in \cite{vdvvz05} that the empirical process of a regular scalar diffusion on an interval $I \subseteq \R$ with finite speed measure 
behaves substantially different.
In fact, weak convergence of the empirical process takes place in $\ell^\infty(\FF)$ if and only if the limit exists as a tight, Borel measurable map.
The proof of this result is heavily based on an analysis of diffusion local time.
One crucial point is the fact that the empirical measure of a univariate regular diffusion is continuous with respect to Lebesgue measure. 
For dimension $d\geq 2$, diffusion local time does not exist, and it is not obvious how to derive uniform limit results under minimal conditions. 
In particular, the empirical measure of a multivariate diffusion is no longer Lebesgue-continuous.

Our first results for the so-called empirical multivariate diffusion process 
$$ \left(\G_t(f)\right)_{f\in\FF}\ :=\ \left(\frac{1}{\sqrt t}\int_0^t f(X_u)\d u\right)_{f\in\FF}$$
parallel to a large extent results from classical empirical process theory. 
However, we find increased regularity for the empirical diffusion process also in higher dimensions. The effect is diminishing with growing dimension but always present. 
To uncover this phenomenon, 
we investigate the subsequent modified version of the empirical process
$$\sqrt t\int_E f(x)\hat\pi_{t,h}(x)\d\lambda(x) 
		= \sqrt t\int_E f(x)\left(\frac{1}{th^d}\int_0^t K\left(\frac{x-X_u}{h}\right)\d u\right)\lambda(\d x),
$$ 
with $\lambda$ denoting the Lebesgue measure.
Smoothed empirical processes based on iid random variables have been investigated by different authors (see, e.g., \cite{yuk92} or \cite{vdv94}).
They found recently particular interest due to the observation that uniform weak convergence of smoothed empirical processes may be deduced even in 
situations when the original empirical process is not tight (Radulovi\'{c} and Wegkamp (2000, 2003, 2009), \cite{gini08}; see also \cite{menzi06}). 
The results however seem to be of limited scope.
For $d=1$ and without further smoothness assumptions on the class $\FF$, \cite{radweg09} require that the bandwidth $h_n$ 
of the kernel estimator satisfies at least $nh_n^2\rightarrow\infty$.
Under slightly different assumptions, \cite{gini08} even need that $nh_n^4\rightarrow\infty$.  
The proofs, based on the decomposition given in Theorem 3.2 in \cite{gizi84}, use closeness of the kernel density estimator $\hat{p}_n$ 
to the underlying density $p$ in a mean-squared sense, 
where $p$ is assumed to be at least differentiable already in the univariate situation. 
Our proof relies on a completely different approach, based on martingale approximation and theory of Markovian semigroups. 
We do not use explicitly the closeness of $\hat{\pi}_{t,h}$ to the Lebesgue density $\pi$ of the invariant measure in a mean squared sense. 
It turns out that uniform weak convergence of the smoothed empirical diffusion process under necessary and sufficient conditions can take place 
with even exponentially small bandwidth in dimension $d=2$, and with strongly un\-der\-smoo\-thing bandwidth choice for pa\-ra\-me\-ters 
$\beta>d/2$ in case $d\geq 3$, assuming that the coordinates of drift and diffusion coefficient belong to some H\"older ball with parameter $\beta$. 
Maybe surprisingly, the performance of the smoothed empirical diffusion process can be guaranteed 
even if the bandwidth is too small for ensuring consistency of $\hat{\pi}_{t,h}$.

\section{Preliminaries}\label{sec_pre}
\subsection{Notation and definitions}

Let $(E,\BB(E))$, $E \subseteq \R^d$, be a Borel measurable space, and let $X=(X_t)_{t\geq 0}$ 
be an $E$-valued Markov process with invariant pro\-ba\-bi\-li\-ty measure $\mu$.
Its transition semigroup is denoted $\left(P_t\right)_{t \geq 0}$ with corresponding transition densities $p_t(\cdot,\cdot)$. 
The infinitesimal generator $A$ of $\left(P_t\right)$ is defined on the domain 
\begin{equation}\nonumber
\left\{f \in \B_0: \left\|\frac{P_tf-f}{t}-g\right\|_{\sup}\rightarrow 0 \text{ for some } g \in \B_0, t \rightarrow 0\right\}
\end{equation}
by $ Af := \lim_{t \rightarrow 0}\left(P_tf-f\right)/t$,
the limit being taken in sup-norm, where
$\B_0 := \{f \in \CC_0: \left\|P_tf-f\right\|_{\sup} \rightarrow 0 \text{ as } t \rightarrow 0\}$ denotes the center of the semigroup, 
and $\CC_0$ is the space of continuous functions $f$ with $f(x)\rightarrow 0$ as $x$ approaches the boundary of $E$.
If the transition probabilities admit an invariant probability measure $\mu$ on $(E,\BB(E))$, then $\left(P_t\right)$ defines a contraction semigroup
on $L^2(E,\d\mu) =: L^2(\d\mu)$ (see \cite{bhat82}). 
With slight abuse of notation, we denote its infinitesimal generator, which is actually an extension of $A$ on $L^2(\d\mu)$, also by $A$, 
with cor\-res\-pon\-ding domain $\mathbb{D}_A\subseteq L^2(\d\mu)$ and range $\RR_A$. 
Together with the norm $\Arrowvert \cdot \Arrowvert_{\D_A}$, given via 
$\Arrowvert g\Arrowvert_{\D_A}^2:=\Arrowvert g\Arrowvert_{\mu,2}^2+\Arrowvert Ag\Arrowvert_{\mu,2}^2$, 
the couple $(\D_A,\Arrowvert \cdot\Arrowvert_{\D_A})$ defines a Banach space (see Section \ref{subsec: inter}).

In case of It\^o--Feller diffusions, the generator $A$ acts on $\CC_K^{\infty}$ as a second-order differential operator, that is
\begin{equation}\label{eq: A}
A_{\arrowvert \CC_K^{\infty}} = \frac{1}{2}\sum_{i,j=1}^d a_{ij}(\cdot)\frac{\partial^2}{\partial x_i\partial x_j} + \sum_{i=1}^d b_i(\cdot)\frac{\partial}{\partial x_i},
\end{equation}
where $a$ and $b$ are a matrix field and a vector field, respectively, on $\R^d$ such that the mappings $x \mapsto a(x)$ and $x\mapsto b(x)$ are
Borel measurable and locally bounded. 
$\CC_K^{\infty}$ denotes the space of infinitely often continuously differentiable functions of compact support in the interior of $E$. 
The matrix $a(x)$ is assumed to be positive definite for any $x$.
The ``carr\'e du champ'' operator on $\CC^2_K(\R^d) \times \CC^2_K(\R^d)$ is defined by
\begin{equation}\label{eq: cdc}
			\Gamma(g,\tilde g) := A(g\tilde g)- g A \tilde g- \tilde g A g = \langle \nabla g,a\nabla \tilde g\rangle.
\end{equation}
Let $1\leq p<\infty$. 
Subsequently, $\WW^{m,p}(\d\mu)$ denotes the Sobolev space of $m$-times weakly differentiable functions $g\in L^p(\d\mu)$ whose weak partial derivatives 
up to order $m$ belong to $L^p(\d\mu)$, too, equipped with the norm
$$
\Arrowvert g\Arrowvert_{\WW^{m,p}(\d\mu)}:=\sum_{\arrowvert\alpha\arrowvert\leq m}\Arrowvert \partial_w^{\alpha}g\Arrowvert_{\mu,p}.
$$
Here, 
$$\partial_w^{\alpha}\ = \frac{\partial_w^{\arrowvert \alpha\arrowvert}}{\partial_wx_1^{\alpha_1}...\partial_wx_d^{\alpha_d}}\ \ \ \text{for all}\ \alpha\in\{0,1,...,m\}^d,
$$
and the derivatives are understood in the weak sense.

Given any initial probability measure $\beta$, let $\P_\beta(\cdot) := \int_E \P_x(\cdot)\beta(\d x)$, denoting with $\P_x$ the law of $X$ starting in $X_0=x$.
Then $\left((X_t)_{t \geq 0}, \P_\mu\right)$ is a stationary  process.
Symmetric diffusions play a special role.
A semigroup $\left(P_t\right)_{t \geq 0}$ is called symmetric with respect to $\mu$, or $\mu$ is called reversible with respect to $\left(P_t\right)_{t \geq 0}$,
if for any $f,g \in L^2(\d\mu)$ 
$$\int_E fP_tg\d\mu = \int_E gP_tf\d\mu.$$
The generator $A$ of a symmetric diffusion is self-adjoint, and it holds $\int_E fA f\d\mu \leq 0$.

\subsection{Continuity of Gaussian processes}
A class of functions $\FF \subset \RR_A$ on $E$ is defined to be Donsker if there exists a tight Borel-measurable random element $\G$ of $\ell^\infty(\FF)$ such that
$\G_t \rightsquigarrow \G$ in $\ell^\infty(\FF)$. 
Here, $\rightsquigarrow$ denotes convergence in law of random elements in the generalized sense of Hoffmann-J\o rgensen (cf. \cite{dud99}, Chapter 3).
Tightness of $\G$ is equivalent to saying that $\G$ admits a version with almost all sample paths bounded and uniformly continuous on $\FF$ with respect to the 
pseudo-metric
\begin{equation}\label{eq_dg}
			d_\G^2(g,\tilde g) := -2 \int_E (g-\tilde g) A(g-\tilde g)\d\mu, \hspace*{7mm} g,\tilde g \in A^{-1}\FF.
\end{equation}
General results on the sample path behavior of Gaussian processes prove useful for deriving Donsker theorems under necessary and sufficient conditions.
Given a separable, infinite-dimensional Hilbert space $H$, a set $C \subset H$ is called a GC-set if the restriction of an isonormal
Gaussian process $L$ on $H$ can be chosen such that its sample functions are uniformly continuous on $C$.
A set $\GG \subset \mathcal{L}_0^2(\P)=\{f\in \mathcal{L}^2(\P): \P(f)=0\}$ is pregaussian if and only if the corresponding set in 
the quotient space $L_0^2(\P)=\{f\in L^2(\P): \P(f)=0\}$ is a GC-set.

\subsection{Assumptions}
We briefly summarize the fundamental assumptions for later purposes. 

(I) \textit{Poincar\'e's inequality.}
The carr\'e du champ $\Gamma$ is said to satisfy a Poincar\'e inequality on $L^2(\d\mu)$ if there exists some constant $c>0$ such that for
all $f \in \D_A$
\begin{equation}\label{eqC.11}
			\Var_\mu(f) := \int f^2 \d\mu - \left(\int f \d\mu\right)^2 \leq c \int \Gamma(f)\d\mu,
\end{equation}
where $\Gamma(f) := \Gamma(f,f)$.
For symmetric diffusions, Poincar\'e's inequality is commonly referred to as spectral gap inequality since (\ref{eqC.11}) 
is then equivalent to the existence of a spectral gap,
$$\lambda_1 := \sup\left\{\lambda \geq 0: E_\lambda - E_0 = 0\right\} = \frac{1}{c_P} > 0,$$
where $-\int_0^\infty \lambda \d E_\lambda$ denotes the spectral decomposition of $A$, and $c=c_P$ appears to be the smallest possible constant in (\ref{eqC.11}). 
Furthermore, (\ref{eqC.11}) is equivalent to the exponential decay of $P_t$ to the invariant measure $\mu$ in $L^2(\d\mu)$ (see, e.g., Theorem 1.3 in \cite{baketal08}),
$$\Var_\mu\left(P_t f\right) \leq \exp\left(-\frac{2t}{c_P}\right)\Var_\mu(f) \qquad \forall \ f \in L^2(\d\mu).$$

(II) \textit{Bound on the transition density.}
There exists some $C_0>0$ such that for any $u\geq t>0$ and for any pair of points $x,y\in\R^d$, satisfying $\Arrowvert x-y\Arrowvert_2^2\leq u$, we have
$$
p_t(x,y)\ \leq\ C_0\left(t^{-d/2}+u^{3d/2}\right).
$$ 

\begin{remark}
In case of a constant diffusion coefficient in a stochastic differential equation, 
assumption (II) holds in particular if the drift satisfies the at most linear growth condition (\cite{qizhe04}, Theorem 3.2, with the choice $q=1+t$ in their notation). 
Additional boundedness of the drift then even allows to drop the $u^{3d/2}$-term (\cite{qian03}, inequality (5)).
\end{remark}

(III) \textit{Uniform ellipticity.} 
There exists some positive constant $\alpha$ such that the differential operator $A$ in (\ref{eq: A}) satisfies the uniform ellipticity condition
$$
\xi\tt a(y)\xi\ \geq \ \alpha\Arrowvert\xi\Arrowvert_2^2 \ \text{for all}\ \xi\in \R^d\setminus \{0\}\ \text{and for all $y\in E$}.
$$

(IV) \textit{Symmetry.} 
The semigroup $(P_t)$ of transition operators is symmetric with respect to $\mu$.

(V) \textit{Invariant density.} 
The invariant probability measure on $(E,\BB(E))$ is Lebesgue continuous with density $\pi$ which is bounded and uniformly bounded away from zero 
on any compact subset of the interior of $E$.

Assumptions (I) -- (V) seem to be rather natural. 
We briefly illustrate them for an example of a diffusion with reflecting boundary conditions. 
For ease of representation, we restrict ourselves to reflecting diffusions on a one-dimensional interval, $[0,1]$, say. 
Consider the stochastic differential equation
$$
\d X_t\ = \ b(X_t)\d t\ +\ \sigma \d W_t\ +\ \nu(X_t)\d l_t,
$$
with $X_t\in [0,1]$ for $t\geq 0$, $W$ a standard Wiener process and $(l_t(X))$ a non-anticipative process that increases only when $X_t\in\{0,1\}$, 
which is part of the solution. 
We assume that $b:[0,1]\rightarrow\R$ is bounded and measurable, $\sigma$ is positive, and that the function $\nu$ satisfies $\nu(0)=1, \nu(1)=-1$. 
The boundedness of $b$ and positivity of $\sigma$ ensure the existence of a weak solution. 
The invariant measure is Lebesgue continuous with density $\pi(x)\d x\sim \sigma^{-2}\exp(-\int_0^x2b(y)/\sigma^2\d y)\d x$. 
Due to the compactness and the reflecting boundary conditions, the corresponding Markov process possesses a spectral gap. 
The associated operator $A$ is self-adjoint and elliptic on $L^2(\d\mu)$ with compact resolvent. 
Its domain $\D_A$ coincides with the subspace of $\WW^{2,2}(\d\mu)$ subject to Neumann boundary conditions. 
Assumptions (I)--(V) are satisfied. 
In order to avoid potential boundary considerations in the sequel, we restrict attention to compactly supported function classes in the interior of $E$ -- even 
if $E$ itself is compact.

\section{Parallels to the classical empirical process}\label{sec_parall}

In the sequel we focus on the special case of symmetric diffusion processes.
Poincar\'e's inequality can be seen as a minimal assumption for the following Bernstein-type inequality due to \cite{lez01}. 
Let $g$ be a bounded measurable function with $\int g \d\mu = 0$.
Then 
\begin{equation}\label{gl2}
		\P_\mu\left(\frac{1}{\sqrt t}\int_0^tg(X_u)\d u>r\right) \ \leq \  \exp\left(-\frac{r^2}{2\left(\sigma^2+c_P\|g\|_{\sup}r/\sqrt t\right)}\right)
			\quad \forall t,r>0,
\end{equation}
where $\sigma^2$ is given as
$$\sigma^2 := \sigma^2(g) := \lim_{t \rightarrow \infty}\frac{1}{t}\Var_{\P_\mu}\left(\int_0^t g(X_u)\d u\right).$$
We now state our first result which parallels Theorem 3.2 in \cite{gizi84}. Note that in contrast to the classical empirical process, the symmetrization technique is not available for the empirical diffusion process. Thus, the method of our proof differs from the proof in \cite{gizi84}. While they randomize in the asymptotic equicontinuity condition in order to apply the comparison inequality due to \cite{fer85}, our proof relies on a random decomposition of the function class $\FF$ and requires in particular the generic chaining bound. Given a function class $\FF$ and any $\delta >0$, let
$$\overline\FF_\delta := \left\{f-g: f,g \in \FF,d_{\G}(f,g) < \delta\right\}.$$

\begin{theorem}\label{theo_preg}
				Let $\left((X_t),\P_\mu\right)$ be a stationary, ergodic diffusion satisfying assumptions (I) -- (V). Denote by $\FF   \subseteq \RR_A\subseteq L_0^2(\d\mu)$ a class of uniformly bounded functions.
				Then $\FF$ is Donsker if and only if it is pregaussian and it holds for any $\eta >0$,
				\begin{equation}\label{eqB.6}
				\E_\mu^* \left\|\frac{1}{\sqrt t}\int_0^t f(X_u)\d u\right\|_{\overline\FF_{\left(\eta/\sqrt t\right)^{1/2}}} \rightarrow 0.
				\end{equation}
\end{theorem}
$\E^*$ and $\P^*$ denote outer expectation and probability, respectively.
\begin{proof}
Closedness of $A$ implies that the null space $N_A := \left\{h \in \D_A: Ah = 0\right\}$ is a closed subset of $L^2(\d\mu)$,
that is, for any $f \in \RR_A$, there exists some unique $g \in \D_A \cap N_A^\perp$ with $Ag=f$. 
For this choice of $g$, we have in particular $\E_{\mu}g(X_0)=0$. 
Convergence of the finite-dimensional marginals follows from martingale approximation 
$$\frac{1}{\sqrt{t}}\Big(g(X_t)-g(X_0)-\int_0^tAg(X_s)ds\Big)
$$ 
and the martingale CLT (see Section 4). It is therefore clear that the above conditions are necessary for $\FF$ to be Donsker. 
Hence, it remains to prove asymptotic equicontinuity. 
Denote by $(\G(f))_{f \in \FF}$ the limiting process whose sample paths are bounded and uniformly continuous with respect to $d_\G$.
Sudakov's minoration (cf. Corollary 3.19 on p. 81 in \cite{letala91}) implies that
$$\lim_{\ep \searrow 0}\ep \sqrt{\log N(\ep,\FF,d_\G)} = 0.$$
For fixed $\eta \geq 0$, define $\delta_t := (\eta/\sqrt t)^{1/2}$ and $m_t := N(\delta_t,\FF,d_\G)$.
Then there exist functions
$f_1, \ldots, f_{m_t} \in \FF$ such that $d_\G(f_i,f_j) > \delta_t$ for all $1\leq i \neq j\leq m_t$ and $\sqrt{\log m_t}\delta_t \rightarrow 0$.
Application of the triangle inequality yields the decomposition 
\begin{equation}\label{decomp}
\sup_{f \in \overline\FF_\eta}\left|\G_t(f)\right| \leq 2\sup_{f \in \overline\FF_{\delta_t}}\left|\G_t(f)\right| 
					+ \max_{1 \leq i \neq j \leq m_t}\left|\G_t(f_i-f_j)\right|,
\end{equation}
which goes originally back to \cite{gizi84}.
We bound the second term in (\ref{decomp}).
For any $f \in (\FF - \FF)$, define
$$X_{f,t} 
:= \frac{1}{\sqrt t}\int_0^t f(X_u)\d u, \qquad \sigma_f^2 := \E_\mu\Gamma({f}), \qquad c_f := \|f\|_{\sup}.$$
We consider the decomposition
\begin{equation}\label{dec1}
		\left|X_{f,t}\right| = 
			\left|X_{f,t}\right|\mathds{1}\Big\{\left|X_{f,t}\right|\leq \frac{\sigma_f^2}{c_f}\sqrt t\Big\} + \left|X_{f,t}\right|\mathds{1}
			\Big\{\left|X_{f,t}\right|> \frac{\sigma_f^2}{c_f}\sqrt t\Big\}.
\end{equation}
It then holds
\begin{equation}\P_\mu\left(\left|X_{f,t}\right| \mathds{1}\Big\{\left|X_{f,t}\right| \leq \frac{\sigma_f^2}{c_f}\sqrt t\Big\} > x\right)
\leq 2\exp\left(-\frac{x^2}{4\sigma_f^2}\right)\label{eq: subgaussian}
\end{equation}
and
\begin{equation}\P_\mu\left(\left|X_{f,t}\right| \mathds{1}\Big\{\left|X_{f,t}\right| > \frac{\sigma_f^2}{c_f}\sqrt t\Big\} > x\right)
\leq 2\exp\left(-\frac{x}{4c_f/\sqrt t}\right).\label{eq: subexponential}
\end{equation}
Thus (cf. the proof of Lemma A.1 in \cite{vdv96}), for some (universal) constant $K$ sufficiently large, not depending on $c_f$ and $\sigma_f$,
$$
		\E_\mu\psi_2\left(\frac{|X_{f,t}|\mathds{1}\left\{|X_{f,t}| \leq \frac{\sigma_f^2}{c_f}\sqrt t\right\}}{K\sigma_f}\right)
		\leq 1, \quad 
		\E_\mu\psi_1\left(\frac{|X_{f,t}|\mathds{1}\left\{|X_{f,t}| > \frac{\sigma_f^2}{c_f}\sqrt t\right\}}{Kc_f/\sqrt t}\right)
		\leq 1,
$$ 
where $\psi_p$ are the Young functions $\psi_p(x) = \exp(x^p)-1$, $p=1,2$.
Then inequality (2.10) in \cite{argi93} implies for $c' := \sup_{f \in \FF}c_f$,
$$\E_\mu \max_{1\leq i\neq j\leq m_t}\left|X_{f_i-f_j,t}\right| 
\mathds{1}\left\{|X_{f_i-f_j,t}| > \frac{\sigma_{f_i-f_j}^2}{c_{f_i-f_j}}\sqrt t\right\} \leq K\log(m_t) \frac{2c'}{\sqrt t}.$$
For estimating the first term in (\ref{dec1}), we use the generic chaining bound in Theorem 1.2.6 in \cite{tala05}, i.e.
$$
\E_\mu\max_{1 \leq i \neq j \leq m_t}\left|X_{f_i-f_j,t}\right|\mathds{1}\left\{|X_{f_i-f_j,t}| \leq \frac{\sigma_{f_i-f_j}^2}{c_{f_i-f_j}}\sqrt t\right\}
 \leq L \gamma_2\left(\overline\FF_\eta, d_\G\right)
$$
for some constant $L>0$, with $\gamma_2$ denoting the $\gamma_2$-functional (cf. \cite{tala05}, Definition 1.2.5).
This completes the verification of the asymptotic equicontinuity condition, noting that $\log(m_t)/\sqrt t \rightarrow 0$ as $t\rightarrow \infty$ by Sudakov's
inequality, while $\gamma_2(\overline\FF_\eta, d_\G) \rightarrow 0$ as $\eta \rightarrow 0$ by pregaussianness.
\end{proof}

The above description of the Donsker property characterizes the effect of pregaussianness on the asymptotic equicontinuity
condition but is of little interest for applications.
Condition (\ref{eqB.6}) remains to be verified, and for the latter purpose 
the concept of VC classes does not seem to be suitable for empirical diffusion processes. 
 Moreover, the result that a function class satisfying a uniform entropy condition is Donsker is proved
by means of symmetrization arguments and is equally not easily transferred to the diffusion setting.
In contrast, 
an analogue of Ossiander's classical bracketing CLT holds.

\begin{theorem}\label{theo_brack}
Let $\left(\left(X_t\right),\P_\mu\right)$ be a stationary, ergodic Feller diffusion satisfying conditions (I) -- (V). 
If $\FF \subseteq L_0^2(\d\mu)$ satisfies
\begin{equation}\label{eq_brack}
				\int_0^\infty \sqrt{\log N_{[\ ]}(\ep,\FF,L^2(\d\mu))}\d\ep < \infty,
\end{equation}
then $\FF$ is Donsker.
\end{theorem}
Here $N_{[\ ]}(\ep,\FF,L^2(\d\mu))$ denotes the $\ep$-entropy with bracketing, that is, the smallest number of $\ep$-brackets (in $L^2(\d\mu)$) which
are required to cover $\FF$ (cf. \cite{vdvw96}, Definition 2.1.6).

\begin{proof} 
Ergodicity yields that $N_A$ is one-dimensional (\cite{bhat82}, Proposition 2.2). 
Since $(P_t)$ is a strongly continuous semigroup on $L^2(\d\mu)$ (implying that $\D_A$ is dense in $L^2(\d\mu)$), 
it follows that $\RR_A= \left\{f \in L^2(\d\mu): \int f\d\mu = 0\right\}$  by the spectral gap inequality (see \cite{bhat82}, Remark 2.3.1). 
Convergence of the finite-dimensional marginals of the empirical process $\left(\G_t(f)\right)_{f \in \FF}$ now follows from martingale 
appro\-xi\-ma\-tion and the martingale CLT, see (\ref{eq: Dynkin}) in Section 4.
It remains to prove asymptotic equicontinuity.
Ossiander's result is about $L^2$-bracketing.
The proof of this classical bracketing CLT as given in \cite{dud99}, pp. 239 -- 244, is based on chaining arguments which are also valid 
when the pseudo-metric $d_\G$ as defined in (\ref{eq_dg}) is used.
The only ingredient of the proof which does not apply in the diffusion context is the classical Bernstein-inequality
which can be replaced with the Bernstein-type inequality (\ref{gl2}).
Furthermore, by Cauchy--Schwarz and Poincar\'e's inequality,
$$
\lim_{t\rightarrow\infty}\mathrm{Var}_{\mu}\big(\G_t(Ag)\big) = \left\langle g,Ag\right\rangle_\mu \leq \|g\|_{\mu,2} \|Ag\|_{\mu,2} \lesssim \|Ag\|_{\mu,2}^2.$$
Here and subsequently, $\lesssim$ means less or equal up to some constant which does not depend on the variable parameters in the expression. 
Since $d_{\G}(A^{-1}f,A^{-1}h)\lesssim\|f-h\|_{\mu,2}$, the bracketing entropy numbers with respect to $d_\G$ can be upper-bounded by $L^2(\d\mu)$-bracketing.
\end{proof}

\begin{remark}
		It is also possible to discretize the empirical process and to work with the discretized version, exploiting some mixing properties.
		In particular, the symmetrization device can be applied after suitable decoupling to the discretized version, such that sufficient conditions ensuring asymptotic equicontinuity 
		(such as Vapnik-Chervonenkis type conditions) can be derived.
		We do not pursue this strategy here but refer the reader to \cite{rio00} for further results in this spirit.
\end{remark}


\section{An intermediate process indexed by smoothed functions}
Although the previous section reveals strong parallels, the work of \cite{vdvvz05} bares substantial differences in the regularity behavior 
of the empirical process of scalar diffusions as compared to the classical empirical process based on independent and identically distributed random variables. 
Indeed, they show for the empirical process of a one-dimensional diffusion with finite speed measure that pregaussianness already implies the Donsker property. 
The explanation lies in the Lebesgue continuity of the empirical measure due to the occupation times formula, i.e. the existence of diffusion local time. 
For higher dimension, diffusion local time does not exist and so this reasoning breaks down.

However, mean-integrated squared error bounds for the estimation of the invariant density at a fixed point in \cite{dalrei07} as compared to those for probability density estimation 
based on independent and identically distributed random variables strongly suggest that there is some increased regularity also in the multidimensional diffusion case.
Thus, one might hope that this effect is getting visible when studying some smoothed version of the empirical diffusion process 
\begin{align}\label{def:S}
(\S_{t,h}(f))_{f\in\FF}\ &:=\ \left(\sqrt t\int_E f(x)\hat\pi_{t,h}(x)\d x\right)_{f\in\FF}\\\nonumber
&=\ \bigg(\sqrt t\int_E f(x)\left(\frac{1}{th^d}\int_0^t K\left(\frac{x-X_u}{h}\right)\d u\right)\lambda(\d x)\bigg)_{f\in\FF},
\end{align}
based on some kernel estimator $\hat\pi_{t,h}$ of the invariant density $\pi$.

For independent and identically distributed random variables $X_1,...,X_n\sim \P$, where $\P$ is some probability measure on $(\R^d,\BB(\R^d))$ with Lebesgue density $p$, 
smoothed empirical processes $(\sqrt{n}\int f(x)\hat{p}_n(x)dx)_{f\in\FF}$ with some density estimator $\hat{p}_n$ have been studied very recently,
having regard to Donsker theorems under necessary and sufficient conditions.
Starting from the hypothesis that $\FF$ is pregaussian, Theorem 3.2 in \cite{gizi84} simplifies the problem of verifying asymptotic equicontinuity to the issue 
of proving some symmetrized analogue of (\ref{eqB.6}), using the closeness of the density estimator to the density in the mean squared sense. 
This has led already in the one-dimensional case to a lower bandwidth bound $nh_n^2\rightarrow\infty$, 
if $\hat{p}_n$ is some kernel density estimator with bandwidth $h_n$. 
Precisely, \cite{radweg09} use the decomposition by the Cauchy--Schwarz inequality,
\begin{align}
\E^*&\sup_{\substack{f\in\FF\\ \P f^2\leq \frac{\eta}{\sqrt{n}}}}\left|\sqrt{n}\int f(x)\big(\hat{p}_n(x)-\E\hat{p}_n(x)\big)\d x\right|\nonumber\\ 
\leq& \ \E^*\sup_{\substack{f\in\FF\\ \P f^2\leq \frac{\eta}{\sqrt{n}}}}
\left(\sqrt{n\P f^2}\right)\Big(\int\frac{\big(\hat{p}_n(x)-\E \hat{p}_n(x)\big)^2}{p(x)}\d x\Big)^{\frac{1}{2}} 
\lesssim \left(\sqrt{n}\int\frac{\mathrm{Var}\,\hat{p}_n(x)}{p(x)}\d x\right)^{\frac{1}{2}}.\label{radweg}
\end{align}
In case of the smoothed empirical diffusion process, this approach is not suitable 
because the constrained set $\overline{\FF}_{\delta}$ in Theorem \ref{theo_preg} is defined by means of $d_{\G}$ 
which is only upper bounded by $\Arrowvert A\cdot \Arrowvert_{\mu,2}$ due to the Poincar\'e inequality. 
Moreover, the above approach needs that the variance decreases to zero slightly faster than $1/\sqrt{n}$. 
Our proof for the smoothed empirical diffusion process is conceptually different. 
We do not use explicitly the closeness of $\hat{\pi}_{t,h}$ to $\pi$ in a mean squared sense. 
Remarkably, the performance of the smoothed empirical diffusion process we consider can be guaranteed 
even if the bandwidth is too small for ensuring consistency of $\hat{\pi}_{t,h}$.




\subsection*{Analyzing the intermediate process}\label{subsec: inter}
In order to circumvent the problematic verification of the equicontinuity condition (\ref{eqB.6}), we study as a preliminary object the 
``empirical process indexed by smoothed functions'' $(\H_{t,h}(g))_{g\in\GG}$, defined by
\begin{equation}\label{def:H}
\H_{t,h}(g) := \G_t\left(A\left(g\ast K_h\right)\right) = \frac{1}{\sqrt{t}}\int_0^tA\left(g\ast K_h\right)(X_u)\d u,
\end{equation}
where $K_h(x):=h^{-d}K( x/h)$ for some compactly supported kernel $K$ on $\R^d$ with $\int K \d\lambda=1$. 
The next proposition guarantees that the process is well-defined on any subspace of the domain $\D_A$ which is locally invariant under translation. We formulate a slightly more precise statement. Note that the result does not need any further specification of $A$ and in particular of $\D_A$ in terms of Sobolev spaces with boundary conditions.\nopagebreak
\begin{proposition}\label{prop: well-defined}
		Assume that $g(\cdot+uh)\in\D_A$ for $h\in [0,h_0],\, u\in\SS^{d-1}$,
		such that $\Arrowvert g(\cdot+uh)\Arrowvert_{\D_A}$ is uniformly bounded in $h\leq h_0,u\in\SS^{d-1}$. 
		Then  $g\ast K_h\in\D_A$ for $h\in[0,h_0]$, and the convolution is contained in the $\|\cdot\|_{\D_A}$-closure of 
		$\|K\|_{\TV}$ times the symmetric convex hull of $\left\{g(\cdot-y):\Arrowvert y\Arrowvert_2\leq h_0\right\}$.
\end{proposition}

Here, $\Arrowvert\cdot\Arrowvert_{\TV}$ denotes the total variation norm. The proof of Proposition \ref{prop: well-defined} is based on the subsequent lemma.

\begin{lemma}\label{lemma: A}
Let $A$ satisfy the Poincar\'e inequality. Then $\left(\D_A,\|\cdot\|_{\D_A}\right)$ is a separable Hilbert space.
\end{lemma}
\begin{proof}
It is clear that $\left(\D_A,\|\cdot\|_{\D_A}\right)$ is pre-Hilbert. 
In order to prove completeness, let $\left(g_n\right)_{n \in \N}$ be a Cauchy sequence in $\left(\D_A,\|\cdot\|_{\D_A}\right)$.
Then $\left(g_n\right)_{n \in \N}$ and $\left(Ag_n\right)_{n \in \N}$ are Cauchy sequences with respect to $\|\cdot\|_{\mu,2}$.
Completeness of $L^2(\d\mu)$ implies that there exist some $g$ such that $\left\|g-g_n\right\|_{\mu,2} \rightarrow 0$ and some $G$
such that $\left\|G-Ag_n\right\|_{\mu,2} \rightarrow 0$. 
Since $A$ is closed, it follows $G = Ag$, and, in particular, $g \in \D_A$.
It remains to prove separability. 
Note that $\RR_A \subset L^2(\d\mu)$ is separable as a subset of a separable metric space.
Let $\left(f_n\right)_{n \in \N}$ be a dense subset of $\RR_A$, and let $\left(g_n\right)_{n\in \N}$ be a dense subset in
$\D_A \cap N_A$, where $N_A$ denotes the null-space of $A$ which is a closed subset of $L^2(\d\mu)$, since $A$ is closed.
For any set $S \subset \D_A$, let $A_{|S}$ denote the restriction of $A$ to $S$.
Then the set
$$\left(A_{|N_A^\perp \cap \D_A}^{-1}(f_n)\right)_{n \in \N} \bigcup \left(g_n\right)_{n \in \N}$$
is countably dense in $\left(\D_A, \|\cdot\|_{\D_A}\right)$.
For the proof, let $g \in \D_A$ be arbitrary. 
Such $g$ can be written as $g = g^\perp + g_0$ for some $g^\perp \in \D_A \cap N_A^\perp$ and some $g_0 \in N_A$.
It holds $\|Ag^\perp\|_{\mu,2} = \|Ag\|_{\mu,2}$. 
Now let 
$$g_k^\perp \subset \left\{A_{|N_A^\perp \cap \D_A}^{-1}(f_n):n \in \N\right\}$$ 
with $\|A(g_k^\perp-g^\perp)\|_{\mu,2}\rightarrow0$. 
Poincar\'e's inequality then gives $\|g_k^\perp - g^\perp\|_{\mu,2} \rightarrow 0$.
Furthermore, let $g_k^0 \subset \left\{g_n: n \in \N\right\}$ such that $\left\|g_k^0-g_0\right\|_{\mu,2} \rightarrow 0$.
Thus, $\left(g_k^\perp + g_k^0\right)$ is the desired approximation.
\end{proof}

\begin{proof}[Proof of Proposition \ref{prop: well-defined}]
By Lemma \ref{lemma: A}, $\D_A$ is closed under finite convex combinations and separable. 
That the convolution $g\ast K_h$ is contained in the $\|\cdot\|_{\D_A}$-closure of $\|K\|_{\TV}$ times the symmetric convex hull of $\left\{g(\cdot-y):\Arrowvert y\Arrowvert_2\leq h_0\right\}$ now follows the lines of the proof of Lemma 1 in \cite{gini08}, replacing $L^2(\d\Q)$ by $\left(\D_A,\|\cdot\|_{\D_A}\right)$. As concerns the first statement of Proposition \ref{prop: well-defined}, it remains to note that the above $\|\cdot\|_{\D_A}$-closure is again contained in $\D_A$ because of the completeness of $\left(\D_A, \|\cdot\|_{\D_A}\right)$ by Lemma \ref{lemma: A}.
\end{proof}

We now present the first result of this section. 
If not explicitly stated otherwise, $K:\R^d\rightarrow \R$ denotes subsequently some twice continuously differentiable kernel of compact support 
which admits the representation $K_{h,x}(z)=h^{-d}\tilde{K}(\Arrowvert x-z\Arrowvert_2/h)$.
Without loss of generality, we assume that the support of $K$ is the closed $d$-dimensional unit ball $B_0(1)$. Subsequently,
$
(\G(f))_{f\in\FF}
$ 
denotes a centered Gaussian process with covariance structure
$$
\mathrm{cov}\big(\G(f),\G(g)\big)\ =\ -\int \Big(\big[A^{-1}f\big]g+\big[A^{-1}g\big]f\Big)\d\mu,\ \ \ f,g\in\FF.
$$
\begin{theorem}\label{thm: smoothed functions}
Let $\left((X_t),\P_\mu\right)$ be a stationary and ergodic It\^o--Feller diffusion in $E\subseteq \R^d$ with non-empty interior $E\setminus\partial E$, 
satisfying the conditions (I)--(V) of Section 2. 
Assume that $A_{|N_A^\perp \cap \D_A}^{-1}\FF$ is a subset of $\WW^{2,2}(\d\mu)$ and possesses a $\mu$-integrable envelope $G$ of compact support $C$ in the interior of $E$. 
Let
$$
\tilde{h}_t^{(d)} := 
					t^{-1/d}\log (\e t)     
$$
If $\FF$ is pregaussian, then 
$$
\left(\H_{t,h_t}\big(A_{|N_A^\perp \cap \D_A}^{-1}f\big)\right)_{f\in\FF} \rightsquigarrow \left(\G(f)\right)_{f\in\FF} \text{ in } \ell^\infty(\FF),$$
 provided that $h_t=h_t^{(d)}\searrow 0$ and $\tilde{h}_t^{(d)}=O(h_t^{(d)})$.
\end{theorem}


\paragraph{\sc Proof} {\it {\large Step I.}} (Convergence of the finite-dimensional distributions)

Recall the definition (\ref{def:H}).
Denote $\GG=A^{-1}_{|N_A^\perp \cap \D_A}\FF\subset \D_A\cap N_A^{\perp}$. Since $X$ is an It\^o--Feller diffusion, $\CC_K^2\subset\D_A$, and we have $g \ast K_h \in \D_A$ for every $g\in\GG$ and $h$ sufficiently small. Hence, $\int A(g \ast K_h) \d\mu = \int Ag\d\mu = 0$.
By Dynkin's formula, $(M^g_t)_{t \geq 0}$ with $M_t^g=g(X_t)-g(X_0)-\int_0^t Ag(X_u)\d u$ is a martingale,
and letting
\begin{equation}\label{eq: Dynkin}
M^g_{t,h_t} := g\ast K_{h_t}(X_t)-g\ast K_{h_t}(X_0)- \int_0^t A (g\ast K_{h_t})(X_u)\d u,
\end{equation}
$\big(t^{-1/2}M_{t,h_t}^g(s)\big)_{0 \leq s \leq t}$ is a triangular array of martingales.
For any natural number $m$, let $g_1,...,g_m$ be some arbitrary collection of elements from $\GG$. 
We will prove that for $h=h_t\searrow 0$, 
\begin{equation}\label{eq: schritt 1}
\left(\begin{array}{l}
	\H_{t,h_t}(g_1)\\ \vdots \\ \H_{t,h_t}(g_m)
\end{array}\right)
=
\left(\begin{array}{l}
	\frac{1}{\sqrt t}M^{g_1}_{t}\\ \vdots \\ \frac{1}{\sqrt t}M^{g_m}_{t}
\end{array}\right)
+ o_\P(1)\ \ \ \text{as $t\rightarrow\infty$.}
\end{equation}
The convergence of the finite-dimensional distributions then follows from the martingale CLT.
For note that by the assumption of the integrable envelope,
\begin{align}
\E_{\mu}^*\sup_{g\in\GG}\left|\H_{t,h_t}(g)-\frac{1}{\sqrt t}M^{g}_{t,h_t}\right|  \ &\leq\ \frac{2}{\sqrt{t}}\E_{\mu}\big(G\ast \arrowvert K_{h_t}\arrowvert (X_0)\big)\nonumber\\
&=\  \frac{2}{\sqrt{t}}\int_E\int G(z-y)\arrowvert K_{h_t}\arrowvert (y)\d \lambda(y)\d\mu(z)\nonumber\\
&\lesssim\ \frac{2}{\sqrt{t}}\int_E\int G(z-y)\arrowvert K_{h_t}\arrowvert (y)\d \lambda(y)\d\lambda(z)\label{eq: equiv1}\\
&=\  \frac{2}{\sqrt{t}}\int\int_E G(x)d\lambda(x)\arrowvert K_{h_t}\arrowvert (y)d\lambda(y) \label{eq: Fubini}\\
&\lesssim\  \frac{1}{\sqrt{t}}\mathbf{E}_{\mu}G(X_0)\ \rightarrow\ 0\ \ \ \text{as $t\rightarrow\infty$}, \label{eq: equiv2}
\end{align}
where we used assumption (V) in (\ref{eq: equiv1}) and (\ref{eq: equiv2}) and Fubini's Theorem 
and the translation invariance of the Lebesgue measure $\lambda$ in (\ref{eq: Fubini}). Furthermore, 
\begin{align}
		\P_\mu\bigg(\max_{i=1,...,m}& \left|\frac{1}{\sqrt t}M_{t,h_t}^{g_i}(t)-\frac{1}{\sqrt t}M_t^{g_i}(t)\right|>\ep\bigg) \nonumber\\
		&\leq\  m \max_{i=1,...,m}\P_\mu\left(\left|\frac{1}{\sqrt t}\left(M^{g_i}_{t,h_t}(t)-M^{g_i}_t(t)\right)\right| > \ep\right)\nonumber\\
		&\leq\ \frac{m}{\ep^2} \max_{i=1,\ldots,m} \E_\mu\left(\frac{1}{\sqrt t}\left(M^{g_i}_{t,h_t}(t)-M^{g_i}_t(t)\right)\right)^2\nonumber\\ 
		&=\ -\frac{m}{\ep^2}\max_{i=1,\ldots,m}2\ \int(g_i-g_i\ast K_{h_t})A\big(g_i-g_i\ast K_{h_t}\big)\d\mu,\label{eq: Bhatt}
\end{align}
where the expression for the variance in (\ref{eq: Bhatt}) is deduced in Bhattacharya (1982). 

As it appears useful for later purposes, we establish at this point the identity
\begin{equation}\label{eq: identity}
-2 \int gAg\d\mu\ =\ \int(\nabla_wg)\tt a(\cdot)\nabla_wg \d\mu 
\end{equation}
for all $g\in\WW^{2,2}(\d\mu)$ of compact support in $E\setminus\partial E$, with the matrix-valued function $a(\cdot)$ of the representation (\ref{eq: A}).
Indeed, the identity (\ref{eq: identity}) is known to be valid for arbitrary $g\in \CC_K^2$, because of $\E_{\mu}A(g^2)=0$ and (\ref{eq: cdc}). 
Let $g\in\WW^{2,2}(\d\mu)$ of compact support in the interior of $E$ be arbitrary, and let $\phi_h(\cdot)=h^{-d}\phi(\cdot/h)$ be a Dirac sequence, 
where $\phi$ is assumed to be compactly supported and twice continuously differentiable. 
Thus, for sufficiently small $h$, $g\ast \phi_h\in \CC_K^2$, and (\ref{eq: identity}) holds true for $g\ast\phi_h$. 
Since the diffusion coefficient matrix is locally bounded, $a$ is uniformly bounded on compacts coordinatewise, hence,
\begin{align}
\int\big(&\nabla\big(g-g\ast \phi_h\big)\big)\tt a\big(\nabla\big(g-g\ast \phi_h\big)\big)\pi\d\lambda \nonumber\\ 
&=\ \int\big(\nabla_wg-(\nabla_wg)\ast \phi_h\big)\tt a\big(\nabla_wg-(\nabla_wg)\ast \phi_h\big)\pi\d\lambda \label{eq: wd1}\\
&= \ o(1)\ \ \text{as $h\searrow 0$}\nonumber
\end{align}
since $\phi_h$ is a Dirac sequence (cf. Theorem 8.14 in \cite{foll99}), where we used Lemma \ref{lemma: wd} in (\ref{eq: wd1}). 
As in Lemma \ref{lemma: sobolev}, one can show that $$-2\,\int gAg\d\mu=\lim_{h\searrow 0}-2\E_{\mu}(g\ast \phi_h)A(g\ast\phi_h),$$ 
which proves the 
identity (\ref{eq: identity}). 

Let $\Arrowvert \cdot\Arrowvert_F$ denote the Frobenius norm.
Using now (\ref{eq: identity}), Lemma \ref{lemma: wd} and the fact that $\sup_{y\in  C}\Arrowvert a(y)\pi(y)\Arrowvert_{F}$ is bounded, expression (\ref{eq: Bhatt}) 
is bounded by\begin{align*}
\lesssim\ \int\big(\nabla_wg_i(x)-(\nabla_wg_{i})\ast K_{h_t}(x)\big)\tt\left(\nabla_wg_i(x)-(\nabla_w g_{i})\ast K_{h_t}(x)\right)\d \lambda(x) &= o(1)
\end{align*}
as $t\rightarrow\infty$, since $( K_{h_t})$ defines a (possibly not non-negatively valued) Dirac sequence. This proves finally (\ref{eq: schritt 1}).

\paragraph{\it Step II} (Asymptotic equicontinuity)

Since 
$$\E_{\mu}^*\sup_{g\in\GG}\left|\H_{t,h_t}(g)-\frac{1}{\sqrt t}M^{g}_{t,h_t}\right| \ \rightarrow \ 0 \text{ as } t \ \rightarrow \ \infty$$
as shown in Step I,
it is sufficient to prove the result for the triangular array of approximating martingales $(t^{-1/2}M^g_{t,h_t})_{g\in\GG}$. Define 
$$
\Arrowvert M_{t,h_t}\Arrowvert_{d_{\G}}\ :=\  \sup_{g,\tilde{g}\in\GG: d_\G(g,\tilde g) > 0}\frac{\sqrt{\langle t^{-1/2}M^g_{t,h_t}-t^{-1/2} M^{\tilde{g}}_{t,h_t}\rangle_t}}{d_{\G}(g,\tilde{g})}.
$$
For any fixed $K>0$ and $\varepsilon>0$,
\begin{align*}
\limsup_{t\rightarrow\infty}\,&\P_{\mu}\bigg(\sup_{d_{\G}(g,\tilde{g})\leq \delta}\frac{1}{\sqrt t}\big\arrowvert M^g_{t,h_t}-M^{\tilde{g}}_{t,h_t}\big\arrowvert > \varepsilon\bigg)\\
&\leq\ \limsup_{t\rightarrow\infty}\,\P_{\mu}\bigg(\sup_{d_{\G}(g,\tilde{g})\leq \delta}\frac{1}{\sqrt t}\big\arrowvert M^g_{t,h_t}-M^{\tilde{g}}_{t,h_t}\big\arrowvert > \varepsilon ; \ 
\Arrowvert M_{t,h_t}\Arrowvert_{d_{\G}} \leq K\bigg)\\ 
&\ \ \ \ \ + \limsup_{t\rightarrow\infty}\,\P_{\mu}\big(
\Arrowvert M_{t,h_t}\Arrowvert_{d_{\G}}> K\big)\\
&=:\ I\ +\ II,\ \ \ \text{say.}
\end{align*}
We first treat expression $I$.
It follows from Bernstein's inequality for continuous local martingales (see, e.g., \cite{revyor99}, p. 153) that
$$\P_\mu \left(\frac{1}{\sqrt t}\left|M^g_{t,h_t}-M^{\tilde g}_{t,h_t}\right| > \ep; \ \left\|M_{t,h_t}\right\|_{d_\G} \leq K\right) \leq 2\exp\left(-
\frac{\ep^2}{2K^2d_\G^2(g,\tilde g)}\right),$$
that is, the random map $g \mapsto \frac{1}{\sqrt t}M^g_{t,h_t}\mathds{1}\big\{\left\|M_{t,h_t}\right\|_{d_\G}\leq K\big\}$ is subgaussian with respect to
$Kd_\G$. Hence by pregaussianness, $\lim_{\delta\searrow 0}\limsup_{t\rightarrow\infty}\,I = 0$.


We now study expression $II$.  Recall the representation of $A$ according to (\ref{eq: A}). By It\^o's formula, 
\begin{align}
\frac{1}{t}\big\langle M_{t,h_t}^g\big\rangle_t\ 
&=\ \frac{1}{t}\int_0^t \left(\left(\nabla_w g\right) \ast K_{h_t}(X_u)\right)\tt a(X_u)\left(\left(\nabla_w g\right) \ast K_{h_t}(X_u)\right)\d \lambda(u)\nonumber\\
&\leq \Arrowvert K\Arrowvert_{\lambda,1}\frac{1}{t}\int_0^t \big(\left(\nabla_wg\right)\tt a(X_u)\left(\nabla_wg\right)\big)\ast \arrowvert K_{h_t}\arrowvert(X_u)\d \lambda(u) \label{eq: qv2}\\
&= \Arrowvert K\Arrowvert_{\lambda,1}\int_{C} \left(\nabla_w g(y)\right)\tt \left(\frac{1}{t}\int_0^t a(X_u)\arrowvert K_{h_t}\arrowvert (y-X_u)\d u\right)\left(\nabla_w g\right)(y)\d \lambda(y),\nonumber
\end{align}
where we used H\"older's inequality in (\ref{eq: qv2}). It remains to verify that
$$
\sup_{g \in \GG}\frac{\int_{ C} \left(\nabla_wg(y)\right)\tt \left(\frac{1}{t}\int_0^t a(X_u)\arrowvert K_{h_t}\arrowvert (y-X_u)\d u\right)\left(\nabla_wg(y)\right)\d \lambda(y)}{\int_E 
\left(\nabla_wg\right)\tt a\left(\nabla_wg\right)\pi\d \lambda} = O_\P(1).$$

Let $x_1,...,x_{N_{h_t}}$ be an $h_t$-net of $C$ with respect to the Euclidean distance, whence $N_h\sim h_t^{-d}$ by compactness of $C$. 
The uniform ellipticity condition and the fact that $\inf_{y\in C}\pi(y)>0$ imply that the last equation holds true whenever
\begin{align}
\sup_{y \in C}&\left\Arrowvert\frac{1}{t}\int_0^t a(X_u)\arrowvert K_{h_t}\arrowvert (y-X_u)\d u\right\Arrowvert_F\nonumber\\
&=\ \sup_{k=1,...,N_{h_t}}\sup_{y\in B_{x_k}(h_t)}\left\Arrowvert\frac{1}{t}\int_0^t a(X_u)\arrowvert K_{h_t}\arrowvert (y-X_u)\d u\right\Arrowvert_F\nonumber\\
&\leq \ \sup_{k=1,...,N_{h_t}}\left\Arrowvert\frac{1}{t}\int_0^t a(X_u)\sup_{y\in B_{x_k}(h_t)}\arrowvert K_{h_t}\arrowvert (y-X_u)\d u\right\Arrowvert_F\nonumber\\
&\leq \ \sup_{k=1,...,N_{h_t}}\left\Arrowvert\frac{1}{t}\int_0^t a(X_u)h_t^{-d}\mathds{1}_{B_{x_k}(2h_t)}(X_u)\d u\right\Arrowvert_F\nonumber\\
&\leq 
\sup_{k=1,...,N_{h_t}}\left\Arrowvert\frac{1}{t}\int_0^t a(X_u)h_t^{-d}\mathds{1}_{B_{x_k}(2h_t)})(X_u)\d u\ - \int a(x)h_t^{-d}\mathds{1}_{B_{x_k}(2h_t)}(x_k-x)\pi(x)\d x\right\Arrowvert_F\label{eq: final}\\
&\ \ \ \  +\ \sup_{k=1,...,N_{h_t}}\left\Arrowvert\int a(x)h_t^{-d}\mathds{1}_{B_{x_k}(2h_t)} (x_k-x)\pi(x)\d x\ \right\Arrowvert_F\nonumber\\
& =\ O_\P(1).\nonumber
\end{align}
But 
$$
\sup_{y \in C}\left\Arrowvert\int a(x) h_t^{-d}\mathds{1}_{B_y(2h_t)} (y-x)\pi(x)\d x \right\Arrowvert_F \lesssim \sup_{y \in  C}\Arrowvert a(y)\Arrowvert_F.
$$
Hence, it remains to prove that expression (\ref{eq: final}) is bounded in probability. This turn out to be a crucial step in the proof. For this purpose, first empirical process techniques are employed which cover any dimension $d\geq 2$. One requisite is Proposition 1 in Dalalyan and \cite{dalrei07}, rephrased for our notation:

\begin{proposition}\label{prop: variance}
	Let $C \subset E\subset\R^d$ be bounded and assume that $\mu \leq \tilde C\lambda$ on $C$ for some positive constant $\tilde C$.  		Then,
			\begin{align*}
\Var_\mu&\left(\frac{1}{\sqrt t}\int_0^t \delta^{-d}\mathds{1}_{B_y(\delta)}(X_u)\d u\right)\  \leq\ 		
			D' \delta^{-2d}\lambda\big(B_y(\delta)\big)^2 \zeta_d^2\big(\lambda(B_1(\delta))\big)
\end{align*}
with
$$
\zeta_d(x) := \begin{cases}
					\max\left\{1, (\log(1/x))^2\right\} & \text{if } d=2,\\
					x^{1/d-1/2}     & \text{if } d\geq 3,
                    \end{cases}
$$
where $D'$ denotes some constant which depends on $c_P, C_0, d$ and $\tilde{C}$ only.
\end{proposition}

First observe that
\begin{equation}\label{eq: zeta_2}
\zeta_d\big(\lambda(B_x(h_t))\big)\ \lesssim\ \zeta_d(h_t^{d}).
\end{equation}
For given fixed $i,j \in \{1,\ldots,d\}$, define for $x\in C$ and $t,h>0$,
\begin{equation}
Z_{x,t}^{h_t}\ :=\ \sqrt t\left(\frac{1}{t}\int_0^t a_{ij}(X_u)
h_t^{-d}\mathds{1}_{B_x(2h_t)}(X_u)\d u - \E_\mu a_{ij}(X_0)h_t^{-d}\mathds{1}_{B_x(2h_t)}(X_0)\right).\label{eq: Z}
\end{equation}
Now, 
(\ref{eq: zeta_2}) and Proposition \ref{prop: variance} imply that
$
		\Var_\mu\big(Z_{x,t}^{h_t}\big) \lesssim \sup_{z \in C}|a_{ij}(z)|^2 \zeta_d^2(h_t^d)
$
for all $x \in C$ and $t>0$.  With 
\begin{align*}
\sigma_{2,h}^2(x)\ :=&\ \lim_{t\rightarrow \infty}\Var_{\P_\mu}(Z_{x,t}^h)\ 
\lesssim\  \sup_{z \in C}|a_{ij}(z)|^2 \zeta_d^2(h^d)
\end{align*}
and
\begin{align*}
c_{\infty,h}\ :=&\  \sup_{z \in C}\big|a_{ij}(z)h^{-d}\mathds{1}_{B_z(2h)}(z)\big|\ \lesssim  h^{-d},
\end{align*}
Lezaud's (2001) Bernstein-type inequality yields the exponential tail bound
$$\P_\mu\left(\left|Z_{x,t}^{h_t}\right|>u\right)\ \leq\  2\exp\left(-\frac{u^2/2}{\sigma_{2,h_t}^2(x)+c_Pc_{\infty, h_t}u/\sqrt t}\right) \qquad \forall u>0.$$
Thus, the same decomposition as in the proof of Theorem \ref{theo_preg} and application of Pisier's maximal inequality to each of the terms subsequently provide the bound
\begin{align}\label{eq: chaining ex}
\E_\mu\sup_{k=1,...,N_{h_t}}&\frac{1}{\sqrt{t}}|Z_{x_k,t}^{h_t}\arrowvert \ \lesssim\  d^4 \sup_{z \in C}|a_{ij}(z)|\frac{1}{\sqrt{t}}\left(\zeta_d(h_t^d)+\frac{\sqrt{\log(eh_t^{-1})}}{h_t^d\sqrt{t}}\right)\sqrt{\log(eh_t^{-1})}. 
\end{align}
Turning back to the proof of Theorem \ref{thm: smoothed functions}, it suffices to note that expression (\ref{eq: final}) 
is bounded in probability whenever $(h_t)$ is chosen as stated in the formulation of Theorem \ref{thm: smoothed functions}, 
due to (\ref{eq: chaining ex}). 
This finally completes the proof.\hfill$\square$

\begin{remark}
	Note that the condition $\sqrt n \Var(\hat p_n(x)) \rightarrow 0$ which is required using the decomposition given in (\ref{radweg}) is connected with the lower bound $h_n \sim n^{-1/2d}$ on the bandwidth for
	the classical smoothed empirical process (the order of the variance is uniform over $x$)
	while in the present Theorem \ref{thm: smoothed functions} the bandwidth $h_t \sim t^{-1/d}\log(\e t)$ is admissible.
\end{remark}

The efficient use of Pisier's inequality requires some kind of uniformity in the tail decay in $t$ and $h_t$, and it is the basic tool for the maximal inequality based on classical chaining techniques. In the present situation, the random variables $Z_{x,t}^{h_t}$ are very localized with increasing sharp maximum for $h_t$ decreasing to zero, i.e. with exploding supremum norm. Thus, although the variances of the increments of the empirical diffusion process are much smaller than those of the classical empirical process, it appears that the use of a refined maximal inequality by $l_2-l_{\infty}$-chaining does not lead to any tighter bound. The results obtained by the techniques of the proof of Theorem \ref{thm: smoothed functions} are not optimal in the one-dimensional case which suggests that improvement is still possible by a different approach. Therefore, we revisit the bound on 
\begin{equation}\label{eq: central}
\sup_{y \in C}\left\Arrowvert\frac{1}{t}\int_0^t a(X_u)\arrowvert K_{h_t}\arrowvert (y-X_u)\d u\right\Arrowvert_F
\end{equation}
in a very specific 
two-dimensional case, employing results from fractal analysis of planar Brownian motion. 

\begin{proposition}\label{prop: dembo}
Let $W$ denote some planar Brownian motion. Let $h_t\searrow 0$ with $\exp(-t)=O(h_t)$. Then
$$
\limsup_{t\rightarrow\infty}\sup_{x\in C}\,\frac{1}{t}\int_0^t\frac{\mathds{1}_{B_x(h_t)}(W_u)}{h_t^2}\d u\ =\ O_\P(1).
$$ 
\end{proposition}

\paragraph{\sc Proof} 
Without loss of generality we assume $C\subset B_0(1)$. By change of variables and the rescaling property of Brownian motion, it holds for any 
$t\geq 1$, $h_t \leq 1$, 
\begin{align*}
\sup_{\Arrowvert x\Arrowvert\leq 1}\int_0^t\frac{\mathds{1}_{B_x(h_t)}(W_u)}{h_t^2(\arrowvert \log h_t\arrowvert   +\frac{1}{2}\log t)}\d u\ &=\ \sup_{\Arrowvert x\Arrowvert\leq 1}t\int_0^1\frac{\mathds{1}_{B_x(h_t)}(W_{tu})}{h_t^2\arrowvert \log h_t  -\frac{1}{2}\log t\arrowvert}\d u\\
&=\ \sup_{\Arrowvert x\Arrowvert\leq 1}\int_0^1\frac{\mathds{1}_{B_{x/\sqrt{t}}(h_t/\sqrt{t})}(t^{-1/2}W_{tu})}{(h_t/\sqrt{t})^2\arrowvert \log (h_t/\sqrt{t})\arrowvert}\d u\\
&\leq_{st.}\ \sup_{r\leq h_t}\sup_{\Arrowvert x\Arrowvert\leq 1}\int_0^1\frac{\mathds{1}_{B_{x}(r)}(W_{u})}{r^2\arrowvert \log r\arrowvert}\d u,
\end{align*}
and the latter term is bounded in probability -- precisely, equals $C+o_\P(1)$ for some constant $C>0$, due to Theorem 1.2, \cite{dematal01}. In particular, $$\sup_{x\in C}\,\frac{1}{t}\int_0^t\frac{\mathds{1}_{B_x(h_t)}(W_u)}{h_t^2}\d u\ =\ O_\P\left(\frac{\arrowvert \log h_t\arrowvert +\arrowvert \log t\arrowvert}{t}\right)\ =\ O_\P(1)$$
if $(h_t)$ is chosen as in the formulation of the proposition. \hfill$\square$

Some well-known technique for proving results about diffusion processes is to reduce the problem via martingale representation theorems to standard Brownian motion.
Thus, if the ergodic, stationary It\^o-Feller diffusion can be represented via some diffeomorphism $\Phi:\R^2\rightarrow\R^2$ from a  local martingale in $\R^2$ satisfying the conditions (I) -- (V), we may replace $\FF$ by $\FF\circ \Phi$ and $(X_t)$ by $(\Phi^{-1}(X_t))$ in the definition of the empirical diffusion process, which makes it sufficient to investigate the intermediate process for the local martingale $(\Phi^{-1}(X_t))$. 
The aim of the next result is to give a hint on further potential improvement of the bound on (\ref{eq: central}) for some specific case of local martingales.

\begin{theorem}
\label{thm: d=2}
Let $\left((X_t),\P_\mu\right)$ be some isotropic, local martingale in $\R^2$ which is representable as an analytic function of planar Brownian motion $B$ 
with $f'\not=0$, i.e.~$X=f(B)$. Let $F(u) := \langle X^{(1)}\rangle_u= \int_0^u\arrowvert  f'(B_s)\arrowvert^2 ds$ and define $\kappa_{C,t,r}$ to be
$$
\arg\max_{\kappa}\left\{\lambda\Big(\Big\{u\in[0,F(t)]: \arrowvert f'(B_{{F^{-1}(u) }})\arrowvert^{-2}_{}\geq \kappa\Big\}\Big)\geq Cr^2(\arrowvert \log r\arrowvert +\log F(t))\right\}
$$
for any $C>0$. 
Let $h_t\searrow 0$ satisfy
\begin{equation}\label{busch}
\frac{1}{h_t^2t}\int_0^{F(t)}\arrowvert f'(B)\arrowvert^{-2}_{F^{-1}(u)}\mathds{1}_{\big\{\arrowvert f'(B)\arrowvert^{-2}_{F^{-1}(u)}>\,\kappa_{
C(1+o_\P(1)),t,h_t}\big\}}\d u\ =\ O_\P(1)
\end{equation}
as $t\rightarrow\infty$ and suppose that conditions (I) -- (V) are satisfied. If
$$A_{|N_A^\perp \cap \D_A}^{-1}\FF\subset\WW^{2,2}(\d\mu)$$ 
possesses a $\mu$-integrable envelope $G$ of compact support $C\subset\R^2$ and $\FF$ is pregaussian, then 
$$
\left(\H_{t,h_t}\big(A_{|N_A^\perp \cap \D_A}^{-1}f\big)\right)_{f\in\FF} \rightsquigarrow \left(\G(f)\right)_{f\in\FF} \text{ in } \ell^\infty(\FF),$$
 provided that $h_t\searrow 0$ and $\exp(-t)=O(h_t)$.
\end{theorem}

\begin{proof}
The proof follows the lines of the proof of Theorem \ref{thm: smoothed functions} until the bound on
(\ref{eq: central}). 
The process $(X_t)$ describes some isotropic, local martingale and possesses therefore  the representation
$
X_t = \tilde{B}\circ \langle X^{(1)}\rangle_t
$ 
with a Dambis--Dubins--Schwarz Brownian motion on some standard extension of $(\FF_t)$. Due to its specific structure, the quadratic variation process can be written as
$
\langle X^{(1)}\rangle_{.} = \int_0^{.}\arrowvert f'(B)\arrowvert^2\d s
$. 
By change of variables,
\begin{align*}
\frac{1}{t}\int_0^t\frac{\mathds{1}_{B_x(h_t)}\big(\tilde{B}_{\langle X^{(1)}\rangle_s}\big)}{h_t^2}\d s \ &=\  \frac{1}{t}\int_0^{\langle X^{(1)}\rangle_t}\arrowvert f'(B)\arrowvert^{-2}_{\langle X^{(1)}\rangle_u^{-1}}\frac{\mathds{1}_{B_{x}(h_t)}(\tilde{B}_u)}{h_t^2}\d u.
\end{align*}
Now, since 
$$
\sup_{\Arrowvert x\Arrowvert\leq 1}\int_0^{F(t)}\mathds{1}_{B_{x}(h_t)}(\tilde{B}_{u})\d u\ \lesssim\ h_t^2\big(\arrowvert \log h_t\arrowvert+\log F(t)\big)(1+o_\P(1))$$
by the proof of Proposition \ref{prop: dembo},
\begin{align*}
\sup_{x\in C}\frac{1}{t}\int_0^{\langle X^{(1)}\rangle_t}&\arrowvert f'(B)\arrowvert^{-2}_{\langle X^{(1)}\rangle_u^{-1}}\frac{\mathds{1}_{B_{x}(h_t)}(\tilde{B}_u)}{h_t^2}\d u\\ 
&\leq\ \frac{1}{h_t^2t}\int_0^{F(t)}\arrowvert  f'(B)\arrowvert^{-2}_{F^{-1}(u)}\mathds{1}_{\big\{\arrowvert f'(B)\arrowvert^{-2}_{F^{-1}(u)}>\,\kappa_{C(1+o_\P(1)),t,h_t}\big\}}\d u
\end{align*}
for some constant $C>0$, which proves the theorem.
\end{proof}

\section{The analysis of the smoothed empirical process}
The remarkable re\-gu\-la\-ri\-ty properties of the process $\H_{t,h}$ clarify the influence of the regularity behavior of the empirical measure for diffusions. 
Instead of following the guideline suggested by Theorem \ref{theo_preg} for the analysis of $\S_{t,h}$, 
it therefore appears useful to relate it to $\H_{t,h}$ under suitable smoothness assumptions on the coefficients of $A$. 
Before discussing the general approximation set-up, we present a simple condition ensuring the equality $A(g\ast K_h)= (Ag)\ast K_h$.

\begin{lemma}\label{lemma: uniformity}
Assume that for some fixed $h_0>0$, 
\begin{equation}\label{eq: uniformity}
\lim_{t\searrow 0}\sup_{z\in B_0(h_0)}
\left\Arrowvert\frac{P_tg(\cdot-z)-g(-z)}{t}-Ag(\cdot-z)\right\Arrowvert_{\mu,2}^2 = 0.  
\end{equation}
Then $g\ast K_h\in\D_A$,and $A(g\ast K_h)=(Ag)\ast K_h$ $\mu$-a.s. ($h\leq h_0$).
\end{lemma}
\begin{proof} 
We will show that under the conditions of the lemma, $(Ag)\ast K_h$ is the $L^2(\d\mu)$-limit of 
$t^{-1}(P_t(g\ast K_h)-g\ast K_h)$ as $t\searrow 0$. 
Using the Jensen inequality and Fubini's Theorem, we have
\begin{align*}
\Big\Arrowvert&\frac{P_t(g\ast K_h)-g\ast K_h}{t}-(Ag)\ast K_h\Big\Arrowvert_{\mu,2}^2\\
&=\ \E_{\mu}\Big(\frac{1}{t}\big[\E\big(g\ast K_h(X_t)\big\arrowvert X_0\big)-g\ast K_h(X_0)\big]-(Ag)\ast K_h(X_0)\Big)^2\\
&=\ \E_{\mu}\bigg(\E\Big(\frac{1}{t}\big[g\ast K_h(X_t)-g\ast K_h(X_0)\big]-(Ag)\ast K_h(X_0)\Big\arrowvert X_0\Big)^2\bigg)\\
&\leq\ \E_{\mu}\Big(\frac{1}{t}\big[g\ast K_h(X_t)-g\ast K_h(X_0)\big]-(Ag)\ast K_h(X_0)\Big)^2\\
&=\ \E_{\mu}\bigg(\int\Big\{\frac{1}{t}\big[g(X_t-z)-g(X_0-z)\big ]-(Ag)(X_0-z)\Big\}K_h(z)\d z\bigg)^2\\
&\leq\  \Arrowvert K\Arrowvert_{\lambda,1}\E_{\mu}\int\Big\{\frac{1}{t}\big[g(X_t-z)-g(X_0-z)\big ]-(Ag)(X_0-z)\Big\}^2\arrowvert K_h\arrowvert(z)\d z\\
&\leq\ \Arrowvert K\Arrowvert_{\lambda,1}^2\sup_{z\in B_0(h_0)}\E_{\mu}\Big(\frac{1}{t}\big(g(X_t-z)-g(X_0-z)\big)-(Ag)(X_0-z)\Big)^2.
\end{align*}
\end{proof}

If no additional information about $\GG$ is available, the applicability of Lemma \ref{lemma: uniformity} is very limited. 
Of course, the identity $A(g\ast K_h)=(Ag)\ast K_h$ holds always true if the coefficients of $A$ are constant in an open set which contains the support of $g$. 
Let $f:E\rightarrow\R$ be a continuous function. 
Then its modulus of continuity on any convex set $K\subset E$ is denoted by
$$
\delta(f,\Delta;K)\ :=\ \displaystyle\sup_{{x,y\in K: \Arrowvert x-y\Arrowvert_2\leq \Delta}}\arrowvert f(x)-f(y)\arrowvert,\ \ 0<\Delta<1.
$$
The next Theorem links the intermediate process $\H_{t,h}$ to the smoothed empirical diffusion process $\S_{t,h}$ via regularity conditions on the local characteristics of the describing stochastic differential equation.

\begin{theorem}\label{thm: smoothed process}
Grant the requirements of Theorem \ref{thm: smoothed functions} on $((X_t), \P_{\mu})$ and $A_{|N_A^\perp \cap \D_A}^{-1}\FF$. 
Assume in addition that the function $G$ of Theorem  \ref{thm: smoothed functions} is also envelope of the spaces of first and second order weak partial derivatives of elements of $A_{|N_A^\perp \cap \D_A}^{-1}\FF$. 
Let $h_t=h_t^{(d)}\searrow 0$ such that  $\tilde{h}_t^{(d)}=O(h_t^{(d)})$. Assume that the characteristics of the stochastic differential equation satisfy the subsequent conditions:

\begin{itemize}
\item[(i)]$\sqrt{t}\max_{i,j}\delta(a_{ij},h_t;C_{\varepsilon}) =  o(1)$,
\item[(ii)]$\sqrt{t}\max_i\delta(b_i,h_t;C_{\varepsilon}) = o(1)$ and 
\item[(iii)] $\sqrt{t} \delta(\pi,h_t;C_{\varepsilon}) = o(1)$ for some $\varepsilon$-neighborhood $C_{\varepsilon}$ of $C$.
\end{itemize}
Then, if $\FF$ is pregaussian,  
$$
\left(\S_{t,h_t}(f)\right)_{f\in\FF}\ \rightsquigarrow\ \left(\G(f)\right)_{f\in\FF}\ \ \ \text{in }\ell^{\infty}(\FF).
$$
\end{theorem}

\begin{proof} 
Recall the definitions (\ref{def:H}) and (\ref{def:S}), and let $\GG := A_{|N_A^\perp \cap \D_A}^{-1}\FF$.
In view of Theorem \ref{thm: smoothed functions} it is sufficient to prove
\begin{equation}\label{eq: approx_1}
\big(\H_{t,h_t}(g)\big)_{g\in\GG}\ =\ \big(\S_{t,h_t}(Ag)\big)_{g\in\GG}\ +\  o_{\P}(1).
\end{equation}
First, 
\begin{align*}
\sqrt{t}\big\arrowvert \E_{\mu}(f\ast K_{h_t})\big\arrowvert \ &=\ \sqrt{t}\big\arrowvert \E_{\mu}(f\ast K_{h_t}-f)\big\arrowvert	\\
&=\ \sqrt{t}\Big\arrowvert \iint_Ef(y)K_{h_t}(x-y)\d y \pi(x)\d x\ -\int_E f(x)\pi(x)\d x\Big\arrowvert\\
&=\ \sqrt{t}\Big\arrowvert \int_Ef(y)\Big(\int_E(\pi(x)-\pi(y))K_{h_t}(x-y)\d x\Big)\d y\Big\arrowvert\\
&\leq \ \Arrowvert K\Arrowvert_{\lambda,1}\Arrowvert f\Arrowvert_{\lambda,1}\sqrt{t}\delta(\pi,h_t;C_{\varepsilon}).					
\end{align*}
The uniform boundedness $
\sup_{f\in\FF} \Arrowvert f\Arrowvert_{L^1}<\infty$ follows from the continuity of $a$ and $b$ and therefore, their uniform boundedness on compacts, the 
existing envelope $G$ on the weak partial derivatives in the expression of $f=Ag$ as well as the equivalence of $\lambda\sim \mu$ on the support of $G$. 
Furthermore, using the abbreviation 
$$
D_w^2g = \left(\frac{\partial_w^2}{\partial x_i\partial x_j}g(\cdot)\right)_{i,j=1}^d,
$$
we obtain
\begin{align*}
\E_{\mu}^*\bigg(\sup_{g\in\GG}\,&\frac{1}{\sqrt{t}}\bigg\arrowvert\int_0^tA(g\ast K_{h_t})(X_u)-(Ag)\ast K_{h_t}(X_u)\d u\bigg\arrowvert\bigg)\\
&\leq\ \sqrt{t}\,\E_{\mu}^*\bigg(\sup_{g\in\GG}\left|A(g\ast K_{h_t})-(Ag)\ast K_{h_t}\right|\bigg)\\
&\leq\ \sqrt{t}\,\int_E\sup_{g\in\GG}\bigg\arrowvert \int \operatorname{tr}\big((a(y)-a(x))\tt D_w^2g(y)\big)K_h(x-y)\d y\bigg\arrowvert \d\mu^*(x) \\ 
&\ \ \ \ \ \ \ \ \ +\  \sqrt{t}\,\int_E\sup_{g\in\GG}\bigg\arrowvert \int \big( b(y)-b(x)\big)^t \nabla_wg(y)K_h(x-y)\d y\bigg\arrowvert \d\mu^*(x)\\
&\lesssim \ \sqrt{t}\big(\max_{i,j}\delta(a_{ij},h_t;C_{\varepsilon})+\max_i\delta(b_i,h_t;C_{\varepsilon})\big)\E_{\mu}\big(G\ast \arrowvert K_{h_t}\arrowvert(X_0)\big)\\
&\lesssim \ \sqrt{t}\big(\max_{i,j}\delta(a_{ij},h_t;C_{\varepsilon})+\max_i\delta(b_i,h_t;C_{\varepsilon})\big)\E_{\mu}G(X_0),
\end{align*}
where the last $\lesssim$ follows by the same reasoning as in (\ref{eq: equiv1}) -- (\ref{eq: equiv2}). These findings entail (\ref{eq: approx_1}).
\end{proof}

\section{Example} We concretize our results for the following specific setting. Suppose that the diffusion is described by the stochastic differential equation
\begin{equation}\label{eq: beispiel}
\d X_t\ =\ b(X_t)\d t\ +\ \sigma \d W_t,\ \ 0\leq t\leq T,
\end{equation}
with drift $b:\R^d\rightarrow\R^d$ and $d$-dimensional Wiener process $W$. 
It is assumed subsequently that there exists a potential $V\in\CC^{1}(\R^d)$ such that $b=-\nabla V$. 
If the drift term $b$ satisfies the at most linear growth condition $\Arrowvert b(x)\Arrowvert_2 \lesssim (1+\Arrowvert x\Arrowvert_2)$, 
then equation (\ref{eq: beispiel}) admits a unique weak solution (\cite{karshr88}, Proposition 3.6). 
If $\exp\circ (-2V)\in L^1(\R^d,\d\lambda)$, there exists a  unique invariant measure (\cite{bhat78}, Theorem 3.5) 
which is Lebesgue continuous with invariant density
$$
\pi(x) =\ \Big(\int_{\R^d}\exp(-2V(u))\d u\Big)^{-1}\exp(-2V(x)),\ \ x\in\R^d,
$$
(\cite{lorber07}, Theorem 8.1.26). 
Suppose further that the process starts in the equilibrium, i.e. $X_0\sim\mu$. 

For any convex set $I\subset \R^d$, let $\HH_d(\beta,L;I)$ denote the isotropic H\"older smoothness class, which for $\beta\leq 1$ equals
$$
\HH_d(\beta,L; I)\ :=\ \Big\{\phi:I\rightarrow\R:\big\arrowvert\phi(x)-\phi(y)\big\arrowvert\leq L\Arrowvert x-y\Arrowvert_2^{\beta}\Big\}.
$$ 
Let $\lfloor \beta\rfloor$ denote the largest integer strictly smaller than $\beta$. 
For $\beta>1$, $\HH_d(\beta,L;I)$ consists of all functions $f:I\rightarrow \R$ that are $\lfloor\beta\rfloor$ times continuously differentiable 
such that the following property is satisfied: 
if $P_y^{(f)}$ denotes the Taylor polynomial of $f$ at the point $y\in I$ up to the $\lfloor\beta\rfloor$-th order, 
$$
\Big\arrowvert f(x)-P_y^{(f)}(x)\Big\arrowvert\ \leq\ L\Arrowvert x-y\Arrowvert_2^{\beta}\ \ \text{for all}\ x,y\in I.
$$

\begin{corollary}\label{das Korollar}
Let $V\in\HH_2(\beta+1,L;\R^2)$ for some $L>0$, $\beta>0$. 
Suppose $b=-\nabla V$ satisfies the at most linear growth condition, $\exp\circ (-2V)\in L^1(\R^2,\d\lambda)$ and
\begin{equation}\label{eq: bedingung}
\max_{i=1,...,d}\max_{\alpha:\arrowvert \alpha\arrowvert\leq \lfloor \beta\rfloor}\big\arrowvert\partial^{\alpha}b_i(0)\big\arrowvert\ \leq\ \gamma
\end{equation}
for some $\gamma>0$. Let $A_{|N_A^\perp \cap \D_A}^{-1}\FF$ satisfy the requirement of Theorem \ref{thm: smoothed process}, where $C\subset E$ is convex. Let one of the following conditions be satisfied:
\begin{itemize}
\item[(i)] $(X_t) = (\Phi(M_t))$ for some isotropic local martingale $(M_t)$ in $\R^2$ satisfying the conditions of Theorem \ref{thm: d=2}, and assume that there exist constants $\delta,\Delta>0$ such that
\begin{equation}
B_x(\delta r)\ \subset\ \Phi^{-1}(B_x(r))\ \subset B_x(\Delta r)\ \ \forall\,x\in C,\ 0<r\leq 1.\label{eq: condIII}
\end{equation}
Suppose that
$
{h}_t :=
					t^{-\eta}
$
for some $\eta>\max\big(1/(2\beta), 1/2\big)$ satisfies condition (\ref{busch}). 
\item[(ii)] Let $h_t\sim t^{-1/2}\log (\e t)$ and $\beta>1$.
\end{itemize}
Then $$
\left(\S_{t,h_t}(f)\right)_{f\in\FF}\ \rightsquigarrow\ \left(\G(f)\right)_{f\in\FF}\ \ \ \text{in }\ell^{\infty}(\FF).$$
\end{corollary}

\begin{proof}
In case of (i), first note that $\sqrt{t}h_t^{\omega}=o(1)$ for arbitrary $\omega\in [\beta,1]$. Now,
\begin{align*}
\sqrt{t}\,\delta\left(b_i,h_t;C_\ep\right)\ &\lesssim\ \begin{cases}
L\sqrt{t}h_t^{\beta} & \text{for $\beta\in (0,1]$,}\\
\sup_{\alpha:\arrowvert\alpha\arrowvert = 1}\sup_{x\in C}\arrowvert \partial^{\alpha}b_i(x)\arrowvert\sqrt{t}h_t
 & \text{for $\beta>1$}                                  \end{cases}\\
&=\ o(1).
\end{align*} 
Furthermore, $V\in\HH_2(\beta+1,L;\R^2)$ and (\ref{eq: bedingung}) 
together imply that $\pi$ is H\"older continuous of order $\beta+1$ on every bounded set $K\subset\R^2$. 
Thus, since $\partial^{\alpha}\pi$ is continuous and hence, uniformly bounded on $C_\ep$ for all $\alpha\in\{0,1\}^2$ with $\arrowvert\alpha\arrowvert = 1$,
\begin{align*}\sqrt{t}\, \delta(\pi,h_t;C_\ep)\ = \ 
		\sqrt{t}\sup_{\substack{x,y\in C_\ep:\\ \Arrowvert x-y\Arrowvert_2\leq h_t}}\big\arrowvert \pi(x)-\pi(y)\big\arrowvert 
		\leq \sqrt{t}\, 2\max_{\alpha:\arrowvert\alpha\arrowvert=1}\sup_{x\in C_\ep}\arrowvert \partial^{\alpha}\pi(x)\arrowvert h_t = o(1).
\end{align*}
The result follows now from Theorem \ref{thm: smoothed process}, using (\ref{eq: condIII}) and noting that
$$
(\ref{eq: central})\ \lesssim\ \sup_{y \in C}\left\arrowvert t^{-1}\int_0^t \mathds{1}_{\Phi^{-1}(B_y(h_t))}(M_u)\d u\right\arrowvert,
$$
which allows the same estimate for this expression as used in Theorem \ref{thm: d=2}. Similar considerations reveal claim (ii).
\end{proof}

\paragraph{\sc Remark} The case of dimension $d=2$ takes an exposited place. 
Here, the possibility of smoothing via even exponentially small bandwidth, though in a very special case, can be possible. 
It is worth being noticed that the H\"older regularity conditions on the local characteristics of the describing stochastic differential equation in 
Corollary \ref{das Korollar} (i) are  much 
weaker than what is typically imposed for ensuring existence and uniqueness of a strong solution.
Note that no specific isotropy structure of the diffusion is required in (ii) to allow $h_t \sim t^{-1/2}\log(\e t)$ which makes the result of Corollary \ref{das Korollar}
applicable in general whenever $\beta > 1$.

Of course, imposing tighter conditions on the modulus of continuity of the local characteristics also enables to establish the approximation (\ref{eq: approx_1}) for the 
$d$-dimensional case with $d\geq 3$, but these are not satisfied for $a$, $b$ and $\pi$ with coordinates belonging to some smoothness class of the H\"older type. While the result in Theorem \ref{thm: smoothed process} does not involve any further regularity constraint on the function class $\FF$, the subsequent Theorem benefits of combined smoothness of the local characteristics and the function class $\FF$ at once. Recall that some function $K:\R^d\rightarrow \R$ is called kernel of order $l$ for some integer $l\geq 0$, if the functions $u\mapsto u_i^jK(u)$, $j=0,1,...,l$ and $i=1,...,d$, are integrable and satisfy
$$
\int K(u)d\lambda(u)\ =\ 1,\ \ \int u_i^jK(u)d\lambda(u)\ =\ 0,\ \ j=1,...,l, \ i=1,...,d.
$$
Here, $u_i$ denotes the $i'$th coordinate of $u\in\R^d$.

\begin{theorem}
Let $V\in\HH_d(\beta+1,L;\R^d)$ for some $L>0$, $\beta>d/2$. 
Suppose $b=-\nabla V$ satisfies the at most linear growth condition, $\exp\circ (-2V)\in L^1(\R^d,\d\lambda)$ and
\begin{equation}\label{eq: bedingung2}
\max_{i=1,...,d}\max_{\alpha:\arrowvert \alpha\arrowvert\leq \lfloor \beta\rfloor}\big\arrowvert\partial^{\alpha}b_i(0)\big\arrowvert\ \leq\ \gamma
\end{equation}
for some $\gamma>0$. 
Assume that  $\GG=A_{|N_A^\perp \cap \D_A}^{-1}\FF$ satisfies the requirement of Theorem \ref{thm: smoothed process}, where $C\subset E$ is convex, $\GG\subset \CC^1(\R^d)$ and 
$$
\big\{\partial^{\alpha} g:\ g\in \GG, \arrowvert \alpha\arrowvert =1, \alpha\in\{0,1\}^d\big\}\subset \HH_d(\beta-1,L; C)
$$
for $\beta>d/2$.
Let 
$
{h}_t^{(d)} :=
					t^{-1/d}\log (et).
$
Then $$
\left(\S_{t,h_t}(f)\right)_{f\in\FF}\ \rightsquigarrow\ \left(\G(f)\right)_{f\in\FF}\ \ \ \text{in }\ell^{\infty}(\FF),
$$
provided that the involved kernel $K$ is of order $2\lfloor\beta\rfloor-1$.

\end{theorem}

\begin{proof}
The proof follows the lines of Theorem \ref{thm: smoothed process} with $h_t=h_t^{(d)}$, involving however different estimates for the expressions
\begin{equation}
\sqrt{t}\big\arrowvert \E_{\mu}(f\ast K_{h_t})\big\arrowvert\ =\ \sqrt{t}\Big\arrowvert \int_Ef(y)\Big(\int_E(\pi(x)-\pi(y))K_{h_t}(x-y)\d x\Big)\d y\Big\arrowvert\label{eq: est 1}
\end{equation}
and
\begin{align}
 \E_{\mu}^*\bigg(\sup_{g\in\GG}\,&\frac{1}{\sqrt{t}}\bigg\arrowvert\int_0^tA(g\ast K_{h_t})(X_u)-(Ag)\ast K_{h_t}(X_u)\d u\bigg\arrowvert\bigg)\nonumber\\ 
&\leq\ \sqrt{t}\,\int_E\sup_{g\in\GG}\bigg\arrowvert \int \big( b(y)-b(x)\big)^t \nabla_wg(y)K_{h_t}(x-y)\d y\bigg\arrowvert \d\mu^*(x),\label{eq: est 2}
\end{align}
respectively. $V\in\HH_d(\beta+1,L;\R^d)$ and (\ref{eq: bedingung2}) 
together imply that $\pi$ is H\"older continuous of order $\beta+1$ on every bounded set $D\subset\R^d$. Since $K$ is of order $\geq d/2$, Taylor expansion of $\pi(y)$ around $\pi(x)$ in the inner integral of the right-hand side in (\ref{eq: est 1}) up to the $\lfloor \beta\rfloor$'th order provides the bound
$$
\sqrt{t}\big\arrowvert \E_{\mu}(f\ast K_{h_t})\big\arrowvert\ \lesssim\ \sqrt{t}h_t^{\beta}\ =\ o(1).
$$

Due to the absolute value involved in (\ref{eq: est 2}), a similar argument based on Fubini's Theorem as in (\ref{eq: est 1}) is not applicable, 
and we need to take advantage of additional smoothness of $\nabla g$ as well. 
Let $b_i$ and $(\nabla g)_i$ denote the $i$'th coordinate of $b$ and $\nabla g$, respectively. 
Let $P_y^{(b_i)}$ denote the Taylor polynomial of $b_i$ at the point $y\in \R^d$ up to the $\lfloor\beta\rfloor$-th order, 
and similarly, $Q_y^{(\nabla g)_i}$ denotes the Taylor polynomial of $(\nabla g)_i$ at the point $y\in \R^d$ up to the order $\lfloor\beta\rfloor -1$. 
Now, with the notation
$$
R_{x,y}^{(b_i)}\, :=\, b_i(y) - P_x^{b_i}(y)\ \ \ \text{and} \ \ \ \tilde{R}_{x,y}^{(\nabla g )_i}\, :=\,  (\nabla g)_i(y)-Q_x^{(\nabla g)_i}(y),
$$
we obtain
\begin{align}
\int \big( b(y)-&b(x)\big)^t (\nabla g)(y)K_h(x-y)\d y\nonumber\\
&=\ \sum_{i=1}^d\int \Big(P_x^{b_i}(y)-b_i(x) + R_{x,y}^{(b_i)}\Big)\Big(Q_x^{(\nabla g)_i}(y) + \tilde{R}_{x,y}^{(\nabla g )_i}\Big) K_h(x-y)\d y\nonumber\\
&= \ \sum_{i=1}^d\int R_{x,y}^{(b_i)}\Big(Q_x^{(\nabla g)_i}(y) + \tilde{R}_{x,y}^{(\nabla g )_i}\Big) K_h(x-y)\d y\label{eq: final I}\\
&\ \ \ \ \ +\ \sum_{i=1}^d\int \Big(P_x^{b_i}(y)-b_i(x)\Big)\tilde{R}_{x,y}^{(\nabla g )_i} K_h(x-y)\d y\label{eq: final II},
\end{align}
since $K$ is of order $2\lfloor \beta\rfloor-1$. As concerns the expression in (\ref{eq: final I}), it remains to note that by the definition of $\HH(\beta,L;C)$,
$$
\sup_{\substack{x,y\in C_{\varepsilon}:\\ \Arrowvert x-y\Arrowvert_2\leq h_t}}\big\arrowvert R_{x,y}^{(b_i)}\big\arrowvert\ \leq\ L\cdot h_t^{\beta},\ \ 
\sup_{\substack{x,y\in C_{\varepsilon}:\\ \Arrowvert x-y\Arrowvert_2\leq h_t}}\big\arrowvert \tilde{R}_{x,y}^{(\nabla g)_i}\big\arrowvert\ \leq\ L\cdot h_t^{\beta-1}.
$$
Furthermore, $Q_x^{(\nabla g)_i}(y)$ is bounded uniformly in $x, y$, since all partial derivatives $(\nabla g)_i$
are bounded in absolute value by the envelope $G$ and $\overline{C_{\varepsilon}}$ is compact. Therefore, $$
\sup_{g\in\GG}\sup_{x,y\in \overline{C_{\varepsilon}}}\arrowvert Q_x^{(\nabla g)_i}(y)\arrowvert < \infty.
$$
With regard to (\ref{eq: final II}), 
$$
\sup_{\substack{x,y\in C_{\varepsilon}:\\ \Arrowvert x-y\Arrowvert_2\leq h_t}}\big\arrowvert  P_x^{b_i}(y)-b_i(x) \big\arrowvert \lesssim h_t,\ \ \ \text{while}\ \ \ \big\arrowvert\tilde{R}_{x,y}^{(\nabla g )_i}\big\arrowvert\ \leq\ L\cdot h_t^{\beta-1}.
$$
Collecting these bounds yields
$
\sup_{g\in\GG}\big\arrowvert \int \big( b(y)-b(x)\big)^t \nabla_wg(y)K_{h_t}(x-y)\d y\big\arrowvert \lesssim h_t^{\beta},
$
that is, (\ref{eq: est 2})$\,=o(1)$.  
\end{proof}

\section{Discussion}
In this article we analyze the empirical diffusion process in higher dimension. 
One motivation for our study was the observation due to \cite{vdvvz05} that it is possible to prove uniform central limit theorems
for empirical processes of scalar regular diffusions with finite speed measure under only pregaussian conditions.
This remarkable result reflects the increased regularity of empirical processes of scalar diffusions due to the existence of
local time, and the analysis of local time is an integral part of its proof.
In higher dimensions, diffusion local time does not exist, and therefore it was not clear at all how to derive 
Donsker theorems under necessary and sufficient conditions.
The question considered in this article is how to bring out some potentially increased regularity of empirical processes
of multidimensional diffusions as compared to classical empirical processes based on iid $\R^d$-valued random variables.

We start by showing that there exist strong parallels to classical empirical process theory.
If (the carr\'e du champ of) the diffusion satisfies Poincar\'e's inequality, then the classical condition of a finite bracketing
entropy integral due to Ossiander can be used for proving Donsker theorems.
Replacing the symmetrization device by an application of the generic chaining bound, 
it is also possible to deduce an analogue of Theorem 3.2 in \cite{gizi84} which describes the effect of pregaussianness on the
asymptotic equicontinuity criterion.
This result simplifies the problem of verifying asymptotic equicontinuity of the empirical diffusion process in the pregaussian setting
as it shows that it suffices to consider the supremum over balls with radius $d_{\G}(f,g) < \delta$ with $\delta =(\eta/\sqrt t)^{1/2}$
for some $\eta >0$.
One subtle difference to the case of classical empirical processes is however that the constraint is formulated in terms of the metric
$d_\G$ instead of a constraint in terms of $\|\cdot\|_{\mu,2}$.
Since Poincar\'e's inequality merely yields an upper bound on $d_{\G}$
in terms of $\Arrowvert A \cdot \Arrowvert_{\mu,2}$, it is not possible to use a decomposition based on the Cauchy--Schwarz inequality
as typically done in the classical situation for verifying the simplified asymptotic equicontinuity criterion.

However, it turned out that some modified version, the intermediate process indexed by smoothed functions, which does not appear naturally in the investigation first
has remarkable regularity properties.
Concerning the proof of this behavior, there are two crucial points.
First, we do not rely mainly on classical empirical process theory but use tools from stochastic analysis
such as martingale approximation relating the empirical diffusion process to the generator of the associated Markovian semigroup. 
In particular, an essential part of the proof is the subgaussian exponential inequality for continuous martingales.
This allows to reduce the proof of uniform weak convergence to deriving an upper bound on the expectation of the supremum of some 
stochastic process which is typically done by means of the chaining technique.
Here again we find substantial differences as compared to the classical situation for the empirical process based on iid random variables; see the discussion below the
proof of Theorem \ref{thm: smoothed functions}.
The outstanding regularity of the intermediate process indexed by smoothed functions in dimension $d=2$ is only detected
by using very sharp results about the occupation measure of planar Brownian motion.
Similar improvement of this type for $d\geq 3$, though less substantial, can presumably be derived again by fractal analysis of multidimensional diffusions.

Furthermore, we would like to stress that for any dimension $d$, already the result of Theorem \ref{thm: smoothed functions} shows a remarkable improvement on the 
restriction concerning the lower bound on the admissible
bandwidth as compared to the smoothed empirical process based on independent and identically distributed random variables in  $\R^d$; see the remark on p.17 and p.3 of the Introduction.

\begin{appendix}
\section*{Appendix}
We collect some elementary functional analytic requisites for our proofs, which might be well-known, 
yet we did not find them anywhere explicitly stated in the present form.

\begin{lemma}\label{lemma: wd}
		Assume that $g \in \WW^{1,2}(\d\mu)$ and $K_h$ is symmetric.
		Then it holds $\mu$-a.s.
		$$
		\partial_w^{\alpha}\left(g \ast K_h\right) = 
		\left(\partial_w^{\alpha}g\right)\ast K_h\ \ \ \text{for all}\ \alpha\in \{0,1\}^d\ \text{with}\ \arrowvert\alpha\arrowvert=1.
		$$
\end{lemma}
\begin{proof}
			By definition of the weak derivative, it holds for all $\phi \in \CC_K^\infty$,
			\begin{eqnarray}\nonumber
			\int \partial^{\alpha}(g\ast K_h)(x)\phi(x)\d \lambda(x)\ &=& -\int\left(g\ast K_h\right)(x)\partial^{\alpha}\phi(x) \d \lambda(x) \\ \nonumber
				&=& - \iint g(y)K_h(x-y)\d y \ \partial^{\alpha}\phi(x)\d \lambda(x)\\\nonumber
				&=& - \iint \partial^{\alpha}\phi(x) K_h(x-y) \d x \ g(y) \d \lambda(y)\\\nonumber
				&=& - \int \left(\partial^{\alpha}\phi \ast K_h\right)(y)g(y)\d \lambda(y)\\\nonumber
				&=& - \int \partial^{\alpha}\left(\phi \ast K_h\right)(y) g(y)\d \lambda(y)\\\label{eq_id}
				&=& \int \partial_w^{\alpha} g(y)  \left(\phi \ast K_h\right)(y) \d \lambda(y)\\ \nonumber
				&=& \int (\partial^{\alpha}_wg)\ast K_h(x)\phi(x)\d \lambda(x),
			\end{eqnarray}
where the identity $\partial^{\alpha}(\phi\ast K_h)=(\partial^{\alpha}\phi)\ast K_h$ in (\ref{eq_id}) is proved, for instance, in Lemma 5 a) in \cite{gini08}.
\end{proof}

\begin{lemma}\label{lemma: sobolev}
Let $A$ be the generator of an It\^o--Feller diffusion in $E\subseteq \R^d$ with continuous drift and diffusion coefficient, $b$ and  $\sigma$, respectively. 
		Then any $g\in\WW^{2,2}(\d\mu)$ of compact support in $E\setminus \partial E$ belongs to $ {\D_A}$, and
		$$
			Ag(\cdot)\ =\ \frac{1}{2}\sum_{i,j=1}^d a_{ij}(\cdot)\frac{\partial_w^2 g}{\partial_w x_i\partial_w x_j}(\cdot)+ 
				\sum_{i=1}^d b_i(\cdot)\frac{\partial_w g}{\partial_w x_i}(\cdot).
		$$
\end{lemma}
\begin{proof} 
	Let $g\in\WW^{2,2}(\d\mu)$ be of compact support in $E\setminus\partial E$.  
	Let $\phi_h(\cdot)=h^{-d}\phi(\cdot/h)$ be a Dirac sequence with some twice continuously differentiable, symmetric kernel $\phi$. 
	Then $g\ast \phi_h$ is twice continuously differentiable and of compact support, hence $g\ast \phi_h\in \D_{A}$ for sufficiently small $h$, 
	and it holds by Lemma \ref{lemma: wd} that
	$$
	\partial^{\alpha}_w(g\ast \phi_h)\ =\ (\partial^{\alpha}_w g)\ast \phi_h\ \ \text{for all multi-indices $\alpha\in\{0,1,2\}^d$ with}\ |\alpha| \leq 2.
	$$
	Therefore,
	$$
		A(g\ast \phi_h)(x)\ =\  \frac{1}{2}\sum_{i,j=1}^d a_{ij}(x)\left(\frac{\partial_w^2g}{\partial_w x_i\partial_w x_j} \ast \phi_h\right)(x)+ 
		\sum_{i=1}^d b_i(x)\left(\frac{\partial_wg}{\partial_w x_i}\ast \phi_h\right)(x).
	$$
	Since $\phi_h$ defines a Dirac sequence, $\left\|g-g\ast \phi_h\right\|_{\mu,2}\rightarrow 0$ as $h\searrow 0$ (cf. Theorem 8.14 in \cite{foll99}). 
	Let 
	$$
	G(x)\ :=\ \frac{1}{2}\sum_{i,j=1}^d a_{ij}(x)\frac{\partial_w^2g}{\partial_w x_i\partial_w x_j}(x)+ 
		\sum_{i=1}^d b_i(x)\frac{\partial_wg}{\partial_w x_i}(x).
	$$
	Denote the union of the supports of $g$ and $g\ast \phi_h$ by $C_h$. 
	Then 
	\begin{align*}
	\Arrowvert A(g\ast \phi_h)-G\Arrowvert_{\mu,2}\ 
		&\leq \ \frac{1}{2}\sum_{i,j=1}^d \left\| a_{ij}\mathds{1}_{C_h}\right\|_{\sup}
				\left\| \frac{\partial_w^2g}{\partial_w x_i\partial_w x_j} \ast \phi_h\ - \frac{\partial_w^2g}{\partial_w x_i\partial_w x_j}\right\|_{\mu,2}\\
		& \hspace*{7mm} +\ \sum_{i=1}^d \left\| b_i\mathds{1}_{C_h}\right\|_{\sup}\left\|\frac{\partial_wg}{\partial_w x_i} \ast \phi_h\ 
				- \frac{\partial_wg}{\partial_w x_i}\right\|_{\mu,2}\\
		&\longrightarrow \ 0 \ \text{as}\ h\searrow 0,
	\end{align*}
	because $\left(\phi_h\right)$ is a Dirac sequence, $g\in\WW^{2,2}(\d\mu)$ and $a(\cdot)$ and $b(\cdot)$ are continuous,  
	hence uniformly bounded on a decreasing sequence of compacts. 
	But this implies $G=Ag$, since $A$ is closed.
\end{proof}
\end{appendix}

\bibliography{donsker_bib}

\begin{thebibliography}{}

\bibitem[Arcones and Gin\'e, 1993]{argi93}
Arcones, M.~A. and Gin\'e, E. (1993).
\newblock {Limit Theorems for $U$-Processes}.
\newblock {\em Ann. Probab.}, 21(3):1494--1542.

\bibitem[Bakry et~al., 2008]{baketal08}
Bakry, D., Cattiaux, P., and Guillin, A. (2008).
\newblock {Rate of convergence for ergodic continuous Markov processes:
  Lyapunov versus Poincar\'e}.
\newblock {\em J. Funct. Anal.}, 245(3):727--759.

\bibitem[Bhattacharya, 1978]{bhat78}
Bhattacharya, R.~N. (1978).
\newblock {Criteria for recurrence and existence of invariant measures for
  multidimensional diffusions}.
\newblock {\em Ann. Probab.}, 6:541--553.

\bibitem[Bhattacharya, 1982]{bhat82}
Bhattacharya, R.~N. (1982).
\newblock {On the functional central limit theorem and the law of the iterated
  logarithm for Markov processes}.
\newblock {\em Z. Wahrscheinlichkeitstheorie verw. Gebiete}, 60:185--201.

\bibitem[Dalalyan and Rei{\ss}, 2007]{dalrei07}
Dalalyan, A. and Rei{\ss}, M. (2007).
\newblock Asymptotic statistical equivalence for ergodic diffusions: the
  multidimensional case.
\newblock {\em Probab. Theory Relat. Fields}, 137(1):25--47.

\bibitem[Dembo et~al., 2001]{dematal01}
Dembo, M., Peres, Y., Rosen, J., and Zeitouni, O. (2001).
\newblock Thick points for planar {B}rownian motion and the {E}rd\"os-{T}aylor
  conjecture on random walk.
\newblock {\em Acta Math.}, 186:239--270.

\bibitem[Dudley, 1999]{dud99}
Dudley, R.~M. (1999).
\newblock {\em Uniform Central Limit Theorems}.
\newblock Cambridge Studies in advanced mathematics. Cambridge University
  Press, Cambridge, UK.

\bibitem[Fernique, 1985]{fer85}
Fernique, X. (1985).
\newblock Sur la convergence {\'e}troite des mesures gaussiennes.
\newblock {\em Z. Wahrscheinlichkeitstheorie Verw. Gebiete}, 68(3):331--336.

\bibitem[Folland, 1999]{foll99}
Folland, G.~B. (1999).
\newblock {\em Real Analysis}.
\newblock John Wiley \& Sons, Inc., New York, 2nd edition.

\bibitem[Gin\'e and Nickl, 2008]{gini08}
Gin\'e, E. and Nickl, R. (2008).
\newblock Uniform central limit theorems for kernel density estimators.
\newblock {\em Probab. Theory Relat. Fields}, 141:333--387.

\bibitem[Gin\'e and Zinn, 1984]{gizi84}
Gin\'e, E. and Zinn, J. (1984).
\newblock Some limit theorems for empirical processes.
\newblock {\em Ann. Probab.}, 14:929--989.

\bibitem[Karatzas and Shreve, 1988]{karshr88}
Karatzas, I. and Shreve, S.~E. (1988).
\newblock {\em Brownian Motion and Stochastic Calculus}.
\newblock Graduate texts in mathematics. Springer, Berlin.

\bibitem[Ledoux and Talagrand, 1991]{letala91}
Ledoux, M. and Talagrand, M. (1991).
\newblock {\em Probability in Banach spaces}, volume~23 of {\em Ergebnisse der
  Mathematik und ihrer Grenzgebiete}.
\newblock Springer, Berlin.

\bibitem[Lezaud, 2001]{lez01}
Lezaud, P. (2001).
\newblock {Chernoff and Berry-Ess\'een inequalities for Markov processes}.
\newblock {\em ESAIM: Probability and Statistics}, 5:183--201.

\bibitem[Lorenzi and Bertoldi, 2007]{lorber07}
Lorenzi, L. and Bertoldi, M. (2007).
\newblock {\em Analytical Methods for Markov Semigroups}.
\newblock Pure and applied mathematics. Chapman \& Hall, New York, first
  edition.

\bibitem[Mendelson and Zinn, 2006]{menzi06}
Mendelson, S. and Zinn, J. (2006).
\newblock Modified empirical {CLT}'s under only pre-{G}aussian conditions.
\newblock In {\em IMS Lecture Notes - Monograph Series. High Dimensional
  Probability}, volume~51, pages 173--184. Institute of Mathematical
  Statistics.

\bibitem[Qian et~al., 2003]{qian03}
Qian, Z., Russo, F., and Zheng, W. (2003).
\newblock Comparison theorem and estimates for transition probability densities
  of diffusion processes.
\newblock {\em Probab. Theory Relat. Fields}, 127(3):388--406.

\bibitem[Qian and Zheng, 2004]{qizhe04}
Qian, Z. and Zheng, W. (2004).
\newblock A representation formula for transition probability densities of
  diffusions and applications.
\newblock {\em Stochastic Process. Appl.}, 111(1):57--76.

\bibitem[Radulovi\'{c} and Wegkamp, 2000]{radweg00}
Radulovi\'{c}, D. and Wegkamp, M. (2000).
\newblock {Weak convergence of smoothed empirical processes: Beyond Donsker
  classes}.
\newblock In {\em High Dimensional Probability II}, volume~47, pages 89--105.
  Birkh\"auser, Boston.

\bibitem[Radulovi\'{c} and Wegkamp, 2003]{radweg03}
Radulovi\'{c}, D. and Wegkamp, M. (2003).
\newblock Necessary and sufficient conditions for weak convergence of smoothed
  empirical processes.
\newblock {\em Statist. Probab. Lett.}, 61(3):321--336.

\bibitem[Radulovi\'{c} and Wegkamp, 2009]{radweg09}
Radulovi\'{c}, D. and Wegkamp, M. (2009).
\newblock {Uniform Central Limit Theorems for pregaussian classes of
  functions}.
\newblock In {\em High Dimensional Probability V: The Luminy Volume}, volume~5,
  pages 84--102. Institute of Mathematical Statistics.

\bibitem[Revuz and Yor, 1999]{revyor99}
Revuz, D. and Yor, M. (1999).
\newblock {\em Continuous Martingales and Brownian Motion}, volume 293 of {\em
  Grundlehren der mathematischen Wissenschaften}.
\newblock Springer, Berlin, 3rd edition.

\bibitem[Rio, 2000]{rio00}
Rio, E. (2000).
\newblock {\em Th\'eorie asymptotique des processus al\'eatoires faiblement
  d\'ependants}.
\newblock Math\'ematiques \& Applications 31. Springer, Berlin.

\bibitem[Talagrand, 2005]{tala05}
Talagrand, M. (2005).
\newblock {\em The Generic Chaining}.
\newblock Springer Monographs in Mathematics. Springer, Berlin.

\bibitem[van~der Vaart, 1994]{vdv94}
van~der Vaart, A. (1994).
\newblock Weak convergence of smoothed empirical processes.
\newblock {\em Scand. J. Statist.}, 21:501--504.

\bibitem[van~der Vaart, 1996]{vdv96}
van~der Vaart, A.~W. (1996).
\newblock New {D}onsker classes.
\newblock {\em Ann. Probab.}, 24(4):2128--2140.

\bibitem[van~der Vaart and van Zanten, 2005]{vdvvz05}
van~der Vaart, A.~W. and van Zanten, H. (2005).
\newblock Donsker theorems for diffusions: Necessary and sufficient conditions.
\newblock {\em Ann. Probab.}, 33(4):1422--1451.

\bibitem[van~der Vaart and Wellner, 1996]{vdvw96}
van~der Vaart, A.~W. and Wellner, J.~W. (1996).
\newblock {\em Weak Convergence and Empirical Processes}.
\newblock Springer Series in Statistics. Springer, New York.

\bibitem[Yukich, 1992]{yuk92}
Yukich, J.~E. (1992).
\newblock Weak convergence of smoothed empirical processes.
\newblock {\em Scand. J. Statist.}, 19:271--279.

\end{thebibliography}
\bibliographystyle{apalike}
\nocite{radweg00}
\nocite{radweg03}

\end{document}